\documentclass[reqno]{amsart}
\usepackage{amsmath,amssymb,amsfonts,caption}
\usepackage{graphicx}
\usepackage{color}
\usepackage[usenames,dvipsnames,svgnames,table]{xcolor}
\definecolor{orange}{rgb}{1,0.5,0}
\usepackage{euscript}
\usepackage{enumerate}
\usepackage{verbatim}
\usepackage{epsfig}

\newtheorem{theorem}{Theorem}
\newtheorem{lemma}{Lemma}

\newtheorem{corollary}{Corollary}
\newtheorem{definition}{Definition}

\newcommand{\eps}{\varepsilon}
\renewcommand{\phi}{\varphi}
\newcommand{\ds}{\, ds}
\newcommand{\dr}{\, dr}

\newcommand{\dt}{\, dt}

\newcommand{\dint}{\displaystyle \int}
\newcommand{\cF}{\mathcal F}
\newcommand{\cV}{\mathcal V}
\newcommand{\cR}{\mathcal R}
\DeclareMathOperator{\codim}{codim}
\DeclareMathOperator{\imagem}{Im}

\DeclareMathOperator{\Id}{Id}
\DeclareMathOperator{\e}{e}

\usepackage{dsfont}
   
   \newcommand{\R}{\ensuremath{\mathds R}}
\begin{document}
\title[Existence of periodic solutions...]
   {Existence of periodic solutions of a periodic SEIRS model with general incidence}
\author{Joaquim P. Mateus}
\address{J. Mateus\\
   Instituto Polit\'ecnico da Guarda\\
   6300-559 Guarda\\
   Portugal}
\email{jmateus@ipg.pt}
\author{C\'esar M. Silva}
\address{C. Silva\\
   Departamento de Matem\'atica\\
   Universidade da Beira Interior\\
   6201-001 Covilh\~a\\
   Portugal}
\email{csilva@mat.ubi.pt}
\urladdr{www.mat.ubi.pt/~csilva}
\date{\today}
\thanks{C. Silva was partially supported by FCT through CMUBI (project UID/MAT/00212/2013) and J. Mateus was partially supported by FCT through UDI (project PEst-OE/EGE/UI4056/2014)}
\begin{abstract} For a family of periodic SEIRS models with general incidence, we prove the existence of at least one endemic periodic orbit when $\cR_0>1$. Additionally, we prove the existence of a unique disease-free periodic orbit, that is globally asymptotically stable when $\mathcal R_0<1$. In particular, our main result establishes a sharp threshold between existence and non-existence of endemic periodic orbits for this family of models.
\end{abstract}
\maketitle
\section{Introduction}

In the sequence of the model introduced by Li and Muldowney in~\cite{Li-Muldowney-MB-1995}, several works were devoted to the study of epidemic models with a latent class. In these models, besides the infected, susceptible and recovered compartments, an exposed compartment is also considered in order to split the infected population in two groups: the individuals that are infected and can infect others (the infective class) and the individuals that are infected but are not yet able to infect others (the exposed or latent class). This division makes the model particularity suitable to include several infectious diseases like measles and, assuming vertical transmission, rubella~\cite{Li-Smith-Wang-SJAM-2001}. Additionally, if there is no recovery, the model is appropriate to describe diseases such as Chagas' disease~\cite{Driessche-ME-2008}. This model can also be used to model diseases like hepatitis B and AIDS~\cite{Li-Smith-Wang-SJAM-2001}. Even influenza can be modeled by a SEIRS model~\cite{Cori-Valleron-E-2012}, although, due to the short latency period, it is sometimes more convenient to use the simpler SIRS formulation~\cite{Edlund-E-2011}. Mathematically, the existence of more than one infected compartment brings some additional challenges to the study of the model.

In this work we focus on the existence and stability of endemic periodic solutions of a large family of periodic SEIRS models contained in the family of models already considered in~\cite{Mateus-Silva-AMC-2014}. Namely, we will consider models of the form
\begin{equation}\label{eq:ProblemaPrincipal}
\begin{cases}
S'=\Lambda(t)-\beta(t)\,\phi(S,N,I)-\mu(t) S+\eta(t)R\\
E'=\beta(t)\,\phi(S,N,I)-(\mu(t)+\epsilon(t))E\\
I'=\epsilon(t)E -(\mu(t)+\gamma(t))I \\
R'=\gamma(t)I-(\mu(t)+\eta(t))R \\
N=S+E+I+R
\end{cases}
\end{equation}
where $S$, $E$, $I$, $R$ denote respectively the susceptible, exposed (infected but not infective), infective and recovered compartments and $N$ is the total population, $\Lambda(t)$ denotes the birth rate, $\beta(t)\,\phi(S,N,I)$ is the incidence into the exposed class of susceptible individuals, $\mu(t)$ are the natural deaths, $\eta(t)$ represents the rate of loss of immunity, $\epsilon(t)$ represents the infectivity rate and $\gamma(t)$ is the rate of recovery. We assume that $\Lambda$, $\beta$, $\mu$, $\eta$, $\epsilon$ and $\gamma$ are periodic functions of the same period $\omega$.

Several different incidence functions have been considered to model the transmission in the context of SEIR/SEIRS models. In particular Michaelis-Menten incidence functions, that include the usual simple and standard incidence functions, have the form $\beta(t)\varphi(S,N,I)=\beta(t) C(N)SI/N$ and were considered, just to name a few references, in~\cite{Wang-Derrick-BIMAS-1978, Bai-Zhou-NARWA-2012, Gao-Chen-Teng-NARWA-2008, Nakata-Kuniya-JMAA-2010, Kuniya-Nakata-AMC-2012, Zhang-Teng-BMB-2007}. The assumption that the incidence function is bilinear is seldom too simple and it is necessary to consider some saturation effect as well as other non-linear behaviors~\cite{Liu-Hethcote-Levin-JMB-1987, Zhou-Xiao-Li-CSF-2007}.
The Holling Type II incidence, given by $\beta(t)\varphi(S,N,I)=\beta(t) SI/(1+\alpha I)$, is an example of an incidence function with saturation effect and was considered for instance in~\cite{Safi-Garba-CMMM-2012, Zhang-Teng-CSF-2009}. Another popular type of incidence, given by $\beta(t)\varphi(S,N,I)=\beta(t) I^pS^q$, was considered in~\cite{Korobeinikov-Maini-MBE-2004, Liu-Hethcote-Levin-JMB-1987, Hethcote-Lewis-Driessche-JMB-1989}. Also, a generalization of Holling Type II incidence,  $\beta(t)\varphi(S,N,I)=\beta(t) S I^p/(1+\alpha I^q)$, was considered in~\cite{Driessche-Hethcote-JMB-1991, Ruan-Wang-JDE-2003}. All these incidence functions satisfy our hypothesis (see~P\ref{Cond-P0}) to~P\ref{Cond-P5}) in Section~\ref{section:NP}).

The search for periodic solutions and the study of their stability is a very important subject in epidemiology. In fact, in the non-autonomous context, periodic solutions play the same role as equilibriums in the autonomous context. Our main result shows that there exists a positive periodic solution of~\eqref{eq:ProblemaPrincipal} whenever $\cR_0>1$ and the determinant of some matrix is not zero, a technical condition required by our method of prove that consists in applying the famous Mawhin continuation theorem. We also prove that, when $\cR_0<1$, there exists a unique disease-free periodic solution that is globally asymptotically stable. Here, $\cR_0$ is given by the spectral radius of some operator, obtained by the method developed in~\cite{Wang-Zhao-JDDE-2008}. We note that in several situations, including mass-action incidence and more generally Michaelis--Menten incidence, we can show that the only condition for the existence of periodic orbit is $\cR_0>1$. We emphasise that, for those incidence functions, $\cR_0$ is a sharp threshold between existence of a (unique and globally attractive) disease-free periodic solution and the co-existence of this disease-free periodic solution with (at least) one endemic periodic solution. To obtain this sharp result, it is fundamental to have a sharp result about persistence of the infectives. Fortunately, in~\cite{Rebelo-Margheri-Bacaer-JMB-2012} such result is obtained for general epidemiological models and applied to a mass-action SEIRS model. We use this result to obtain persistence in our general incidence case.

For mass-action incidence, in~\cite{Zhang-Liu-Teng-AM-2012}, it is discussed the existence of periodic orbits. It is shown there that, under some condition involving bounds for the periodic parameters, there exists at least a positive periodic orbit. The referred model differs from ours not only because it assumes a particular form for the incidence function, but also because it allows disease induced mortality and it assumes that immunity is permanent. When the disease induced mortality is set to zero (letting $\alpha\equiv 0$), that model becomes a particular case of ours. Thus, when there is no disease induced mortality, Corollary~\ref{cor:simple} in Section~\ref{section:PersistenceExistence} improves the main result in~\cite{Zhang-Liu-Teng-AM-2012}.

Although the idea of applying Mawhin's continuation theorem was borrowed from~\cite{Zhang-Liu-Teng-AM-2012}, we need several nontrivial new arguments to deal with our case. In particular, because we allow temporary immunity, we were forced to use the original four-dimensional system instead of a reduced system.
\section{Notation and Preliminaries}\label{section:NP}

In this section we will establish the assumptions on model~\eqref{eq:ProblemaPrincipal} and state some results on threshold type conditions obtained in~\cite{Mateus-Silva-AMC-2014} for this model.

Given an $\omega$-periodic function $f:\R_0^+\to \R$, we define $f^u=\max_{t \in [0,\omega]} f(t)$ and $f^\ell=\min_{t \in [0,\omega]} f(t)$. We will make the following assumptions:
\begin{enumerate}[P1)]
\item \label{Cond-P0} There is $\omega \ge 0$ such that $\Lambda$, $\mu$, $\beta$ and $\epsilon$ are continuous, bounded and positive $\omega$-periodic real valued functions on $\R_0^+$ and that $\eta$ and $\gamma$ are continuous, bounded and non-negative $\omega$-periodic real valued functions on $\R_0^+$;
\item \label{Cond-P1} Function $\varphi:(\R_0^+)^3 \to \R$ is continuously differentiable;
\item \label{Cond-P2} For $S,N,I \ge 0$ we have $\varphi(0,N,I)=\varphi(S,N,0)=0$;
\item \label{Cond-P3} For $S,I>0$ and $N \in \left[\Lambda^\ell/\mu^u,\Lambda^u/\mu^\ell\right]$ we have
    $c_1 \le \varphi(S,N,I)/(SI) \le c_2$;
\item \label{Cond-P4} For $0\le I\le N\le\Lambda^u/\mu^\ell$, the function $\R^+_0 \ni S \mapsto \varphi(S,N,I)$ is non-decreasing, for $0\le S \le N\le\Lambda^u/\mu^\ell$, the function $\R^+_0 \ni I \mapsto \varphi(S,N,I)$ is non-decreasing and for $0\le S,I\le N \le \Lambda^u/\mu^\ell$ the function $\R^+_0 \ni N \mapsto \varphi(S,N,I)$ is non-increasing;
\item \label{Cond-P5} For $0\le S\le N\le\Lambda^u/\mu^\ell$, the function $\R^+ \ni I \mapsto \varphi(S,N,I)/I$ is non-increasing.
\end{enumerate}

We will consider in our periodic setting the periodic linear differential equation
\begin{equation}\label{eq:SistemaAuxiliar-periodic}
z'=\Lambda(t)-\mu(t)z.
\end{equation}
We have the following proposition:

\begin{lemma}\label{lemma:exist-per-sol-DF}
Assume that condition P\ref{Cond-P0}) holds. Then we have the following:
\begin{enumerate}[1)]
    \item \label{cond-1-aux} Given $t_0 \ge 0$, all solutions $z$ of equation~\eqref{eq:SistemaAuxiliar-periodic} with initial condition $z(t_0) \ge 0$ are nonnegative for all $t \ge 0$;
    \item \label{cond-1a-aux} Given $t_0 \ge 0$, all solutions $z$ of equation~\eqref{eq:SistemaAuxiliar-periodic} with initial condition $z(t_0) > 0$ are positive for all $t \ge 0$;
    \item \label{cond-2-aux} Given any two solutions $z,z_1$ of~\eqref{eq:SistemaAuxiliar-periodic} we have $|z(t)-z_1(t)| \to 0$ as $t \to +\infty$;
    \item \label{lemma:exist-per-sol-DF-1} For each solution $z(t)$ of~\eqref{eq:SistemaAuxiliar-periodic} we have
      $$\Lambda^\ell/\mu^u \le \liminf_{t \to +\infty} z(t) \le \limsup_{t \to +\infty} z(t) \le \Lambda^u/\mu^\ell;$$
    \item \label{lemma:exist-per-sol-DF-2} For each solution $z(t)$ of~\eqref{eq:SistemaAuxiliar-periodic} with initial condition in $[\Lambda^\ell/\mu^u, \Lambda^u/\mu^\ell]$ we have $z(t) \in [\Lambda^\ell/\mu^u, \Lambda^u/\mu^\ell]$, for all $t \ge t_0$;
    \item \label{lemma:exist-per-sol-DF-3} There is a unique periodic solution $z^*(t)$ of~\eqref{eq:SistemaAuxiliar-periodic} in $\R^+$, this solution has period $\omega$ and is given by
      \begin{equation}\label{eq:aux-sist-per-sol}
      z^*(t)=\dfrac{\int_0^\omega \Lambda(u)\e^{-\int_u^\omega \mu(s) \ ds} \ du}{1-\e^{-\int_0^\omega \mu(s)\ds}}\e^{-\int_0^t \mu(s)\ds}+\int_0^t \Lambda(u)\e^{-\int_u^t \mu(s) \ ds} \ du.
      \end{equation}
\end{enumerate}
\end{lemma}

\begin{proof}
By the variation of the parameters formula we have that the unique solution of~\eqref{eq:SistemaAuxiliar-periodic} with $z(0)=z_0$ is
\begin{equation}\label{eq:var-consta}
z(t,z_0)=z_0\e^{-\int_0^t \mu(s)\ds}+\int_0^t \Lambda(u)\e^{-\int_u^t \mu(s) \ ds} \ du.
\end{equation}
and thus we immediately conclude that, if $z_0\ge0$, then $z(t,z_0)\ge0$ for all $t\ge 0$ and similarly that, if $z_0>0$, then $z(t,z_0)>0$ for all $t\ge 0$. Thus,~\ref{cond-1-aux}) and \ref{cond-1a-aux}) hold.

Letting $z$ and $z_1$ be solutions of~\eqref{eq:SistemaAuxiliar-periodic} with $z(t_0)=z_0$ and $z_1(t_0)=z_{0,1}$, by~\eqref{eq:var-consta} and~P\ref{Cond-P0}), there is $t_0>0$ such that, for $t \ge t_0$ we have
	\[
   |z(t)-z_1(t)| = \e^{-\int_{t_0}^t\mu(s)\ds} |z_0-z_{0,1}| \le \e^{-\mu^\ell (t-t_0)} |z_0-z_{0,1}|
   \]
and thus $|z(t)-z_1(t)|\to 0$ as $t \to +\infty$ and we obtain~\ref{cond-2-aux}).

To obtain bounds for the solutions we note that
\begin{equation}\label{eq:maj-sist-aux-1}
z(t,z_0) \le z_0 \e^{-\mu^\ell t} + \Lambda^u \int_0^t \e^{-\mu^\ell(t-u)} \ du = \left(z_0-\Lambda^u/\mu^\ell \right) \e^{-\mu^\ell t} + \Lambda^u/\mu^\ell
\end{equation}
and thus $\displaystyle \limsup_{t \to +\infty} z(t,z_0) \le \Lambda^u/\mu^\ell$. Similarly,
\begin{equation}\label{eq:maj-sist-aux-2}
z(t,z_0) \ge z_0 \e^{-\mu^u t} + \Lambda^\ell \int_0^t \e^{-\mu^u(t-u)} \ du = \left(z_0-\Lambda^\ell/\mu^u \right) \e^{-\mu^u t} + \Lambda^\ell/\mu^u
\end{equation}
and thus $\displaystyle \liminf_{t \to +\infty} z(t,z_0) \ge \Lambda^\ell/\mu^u$. We obtain~\ref{lemma:exist-per-sol-DF-1}).

Let $z_0 \in [\Lambda^\ell/\mu^u, \Lambda^u/\mu^\ell]$. Since $z_0-\Lambda^u/\mu^\ell\le 0$ and $z_0-\Lambda^\ell/\mu^u\ge 0$, by~\eqref{eq:maj-sist-aux-1} and~\eqref{eq:maj-sist-aux-2} we obtain~\ref{lemma:exist-per-sol-DF-2}).

By the invariance of $[\Lambda^\ell/\mu^u, \Lambda^u/\mu^\ell]$ established in~\ref{lemma:exist-per-sol-DF-2}), the map
\[
\begin{array}{cccc}
P: & [\Lambda^\ell/\mu^u, \Lambda^u/\mu^\ell] & \to & [\Lambda^\ell/\mu^u, \Lambda^u/\mu^\ell]\\
 & y & \to & z(\omega,y)
\end{array},
\]
where $z(t,y)$ denotes the unique solution of~\eqref{eq:SistemaAuxiliar-periodic} with initial condition $z(0)=y$, is well defined. Since $P$ is a continuous function on the convex and compact set $[\Lambda^\ell/\mu^u, \Lambda^u/\mu^\ell]$, by Brower's fixed point theorem, we conclude that $P$ has a fixed point $y_0$. Thus $z(\omega,y_0)=y_0$. By uniqueness of solution we have
\begin{equation}\label{eq:aux-sist-uniqueness-per-sol}
z(\omega+t,y_0)=z(t,z(\omega,y_0))=z(t,y_0)
\end{equation}
and we can conclude that the solution of~\eqref{eq:SistemaAuxiliar-periodic} with $z(0)=y_0$ is a $\omega$-periodic solution. Moreover, by~\eqref{eq:aux-sist-uniqueness-per-sol} and the variation of the parameters formula, we obtain
$$y_0=y_0\e^{-\int_0^\omega \mu(s)\ds}+\int_0^\omega \Lambda(u)\e^{-\int_u^\omega \mu(s) \ ds} \ du$$
and thus
\[
y_0=\dfrac{\int_0^\omega \Lambda(u)\e^{-\int_u^\omega \mu(s) \ ds} \ du}{1-\e^{-\int_0^\omega \mu(s)\ds}}
\]
and we get~\eqref{eq:aux-sist-per-sol}. The uniqueness of the periodic solution follows from the global asymptotic stability of solutions proved in~\ref{cond-2-aux}). We obtain~\ref{lemma:exist-per-sol-DF-3}).
\end{proof}

We now obtain some simple properties of system~\eqref{eq:ProblemaPrincipal}.
\begin{lemma}\label{lemma:invariant-region-periodic}
Assume that conditions P\ref{Cond-P0}) to P\ref{Cond-P5}) hold. Then:
\begin{enumerate}[1)]
\item \label{cond-1-bs} All solutions $(S(t),E(t),I(t),R(t))$ of~\eqref{eq:ProblemaPrincipal}
with nonnegative initial conditions,
$S(0),E(0),I(0),R(0) \ge 0$,
are nonnegative for all $t \ge 0$;
\item \label{cond-2-bs} All solutions $(S(t),E(t),I(t),R(t))$ of~\eqref{eq:ProblemaPrincipal}
with positive initial conditions, $S(0)$, $E(0)$, $I(0)$, $R(0)>0$, are positive for all $t \ge 0$;
\item \label{cond-a-irp} If $(S(t),E(t),I(t),R(t))$ is a periodic solution of~\eqref{eq:ProblemaPrincipal} verifying $S(t_0)$, $E(t_0)$, $I(t_0)$, $R(t_0) \ge 0$, then we have
$\Lambda^\ell/\mu^u \le N(t) \le \Lambda^u/\mu^\ell$.
\item \label{cond-2-irp} For any $\delta>0$, and every solution $(S(t),E(t),I(t),R(t))$, there is $T_\delta>0$ such that $(S(t),E(t),I(t),R(t))$ belongs to the set
    $$\left\{(S,E,I,R) \in (\R_0^+)^4: \Lambda^\ell / \mu^u-\delta \le S+E+I+R \le \Lambda^u/ \mu^\ell+\delta \right\},$$
    for all $t\ge T_\delta$.
\end{enumerate}
\end{lemma}

\begin{proof}
A simple analysis of the flow on the boundary of $(\R_0^+)^4$ allows one to conclude that~\ref{cond-1-bs}) and~\ref{cond-2-bs}) hold.
To obtain the remaining conditions we note that, adding the differential equations in~\eqref{eq:ProblemaPrincipal} we get the equation $N'=\Lambda(t)-\mu(t)N$. By Lemma~\ref{lemma:exist-per-sol-DF}, we easily obtain~\ref{cond-a-irp}) and \ref{cond-2-irp}).
\end{proof}

By P\ref{Cond-P0}) and P\ref{Cond-P1}), the right end side of our system is continuous and locally Lipschitz and thus, by Picard--Lindel\"of's theorem we have existence and uniqueness of (local) solution. By~\ref{cond-2-irp}) in Lemma~\ref{lemma:invariant-region-periodic}, every solution is global in the future.

\section{Existence and stability of disease-free periodic orbits}

\begin{theorem}\label{eq:DFperiodicSol}
Assume that conditions P\ref{Cond-P0}) to P\ref{Cond-P5}) hold. Then system~\eqref{eq:ProblemaPrincipal} admits a unique disease-free periodic solution given by $x^*=(S^*(t),0,0,0)$, where $S^*$ is the unique periodic solution of~\eqref{eq:SistemaAuxiliar-periodic}. This solution has period $\omega$.
\end{theorem}

\begin{proof}
By Lemma~\ref{lemma:exist-per-sol-DF}, equation
$$S'=\Lambda(t)-\mu(t)S$$
with initial condition $S(0)=S_0>0$ admits a unique positive periodic solution $S^*(t)$, which is globally attractive. Since $R'=-(\mu(t)+\eta(t))R$ has general solution $R(t)=C\e^{-\int_0^t \mu(s)+\eta(s) \ ds}$, we conclude that for any periodic solution we must have $C=0$. Thus system~\eqref{eq:ProblemaPrincipal} admits an unique disease-free periodic solution given by $(S^*(t),0,0,0)$.
Since $S^*(t)$ is $\omega$-periodic, it follows that $(S^*(t),0,0,0)$ is $\omega$-periodic.
\end{proof}

To obtain the basic reproductive number, we will use the general setting and the notation in~\cite{Wang-Zhao-JDDE-2008} and, letting $x=(x_1,x_2,x_3,x_4)=(E,I,S,R)$, we can write system~\eqref{eq:ProblemaPrincipal} in the form
    $$x'=\mathcal F(t,x)-(\mathcal V^-(t,x)-\mathcal V^+(t,x))$$
where
\[
\cF(t,x)
=
\left[
\begin{array}{c}
\beta(t) \varphi(S,N,I)\\ 0\\ 0\\ 0
\end{array}
\right],
\]
\[
\cV^-(t,x)
=
\left[
\begin{array}{c}
(\mu(t)+\epsilon(t))E \\ (\mu(t)+\gamma(t))I\\
\beta(t)\varphi(S,N,I)+\mu(t) S\\ (\mu(t)+\eta(t)) R
\end{array}
\right]
\]
and
\[
\cV^+(t,x)
=
\left[
\begin{array}{c}
0\\  \eps E\\ \Lambda(t)+\eta(t) R\\ \gamma(t)I
\end{array}
\right].
\]
It is easy to see that conditions~(A1) to~(A5) in page 701 of~\cite{Wang-Zhao-JDDE-2008} are consequence of conditions P\ref{Cond-P0}) to P\ref{Cond-P5}).

Letting $x^*=(0,0,S^*(t),0)$ be the unique positive $\omega$-periodic solution of~\eqref{eq:ProblemaPrincipal} given by Theorem~\ref{eq:DFperiodicSol}, by P\ref{Cond-P1}) and P\ref{Cond-P2}) we have $\frac{\partial \varphi}{\partial N}(S^*(t),S^*(t),0)$=0 and therefore the matrices in (2.2) in~\cite{Wang-Zhao-JDDE-2008} are given by
\[
F(t)=
\left[
\begin{array}{cc}
0  & \beta(t)\frac{\partial \varphi}{\partial I}(S^*(t),S^*(t),0)\\
0 & 0
\end{array}
\right]
\]
and
\[
V(t)=
\left[
\begin{array}{cc}
\mu(t)+\eps(t) & 0\\
-\eps(t) & \mu(t)+\gamma(t)
\end{array}
\right].
\]
Denote by $Y(t,s)$, $t\ge s$, the evolution operator of the linear $\omega$-periodic system $y'=-V(t)y$, i.e. $Y(t,s)$ is such that
\[
  \dfrac{d}{dt}[Y(t,s)]
  =
\left[
\begin{array}{cc}
-(\mu(t)+\eps(t)) & 0\\
\eps(t) & -(\mu(t)+\gamma(t))
\end{array}
\right]
  Y(t,s)
\]
for $t \ge s$, $s\in \R$.  The next infection operator $L:C_\omega \to C_\omega$ becomes in our context
$$(L\varphi)(t)=\int_0^\infty Y(t,t-a)F(t-a)\varphi(t-a)\ da$$
and we define the basic reproduction ratio in our context by
$$\mathcal R_0=\rho(L).$$

By Theorem~2.2 in~\cite{Wang-Zhao-JDDE-2008} we get the following result.

\begin{theorem}\label{local-stability-periodic-seirs}
Assume that conditions P\ref{Cond-P0}) to P\ref{Cond-P5}) hold. Then, for system~\eqref{eq:ProblemaPrincipal}, the disease-free periodic solution $x_0^*$ is locally asymptotically stable if $\mathcal R_0<1$ and unstable if $\mathcal R_0 > 1$. Furthermore
\begin{enumerate}[1)]
  \item $\mathcal R_0=1$ if and only if $\rho(\Phi_{F-V}(\omega))=1$;
  \item $\mathcal R_0<1$ if and only if $\rho(\Phi_{F-V}(\omega))<1$;
  \item $\mathcal R_0>1$ if and only if $\rho(\Phi_{F-V}(\omega))>1$,
\end{enumerate}
where $\Phi_{F-V}(t)$ is the fundamental matrix solution of the linear system $$x'=(F(t)-V(t))x.$$
\end{theorem}

We begin by defining some concepts. Let $A$ be an square matrix. We say that $A$ is cooperative if all its off-diagonal elements are non-negative and we say that $A$ is irreducible if it can not be placed into block upper-triangular form by simultaneous row/column permutations.
To obtain the global stability of the disease-free periodic solution we need an auxiliary result.

\begin{lemma}[Lemma 2.1 in~\cite{Nakata-Kuniya-JMAA-2010}]\label{lemma:global-stab-dfe-aux}
Let $A(t)$ be a continuous, cooperative, irreducible and $\omega$-periodic matrix function, let $\Phi_A(t)$ be the fundamental matrix solution of \begin{equation}\label{eq:lin-equation-periodic-persist}
x'=A(t)x
 \end{equation}and let $p=\frac{1}{\omega}\ln(\rho(\Phi_A(\omega)))$, where $\rho$ denotes the spectral radius. Then, there exists a positive $\omega$-periodic function $v(t)$ such that $\e^{pt}v(t)$ is a solution of~\eqref{eq:lin-equation-periodic-persist}.
\end{lemma}

We are now in conditions to state a result about the persistence of the infectives in our context.

\begin{theorem}
If conditions P\ref{Cond-P0}) to P\ref{Cond-P5}) hold, the disease-free $\omega$-periodic solution $x^*=(S^*(t),0,0,0)$ of system~\eqref{eq:ProblemaPrincipal} is globally asymptotically stable if $\mathcal R_0<1$.
\end{theorem}

\begin{proof}
By Theorem~\ref{local-stability-periodic-seirs}, if $\cR_0<1$, then $x^*(t)=(S^*(t),0,0,0)$, the disease-free $\omega$-periodic solution, is locally asymptotically stable. On the other hand, by~\ref{cond-2-aux}) in Lemma \ref{lemma:exist-per-sol-DF}, for any $\eps_1>0$ there exists $T_1>0$ such that
\begin{equation}\label{eq:bound-N}
S^*(t)-\eps_1 \le N(t) \le S^*(t)+\eps_1
\end{equation}
for $t > T_1$. Thus $S(t) \le N(t) \le S^*(t)+\eps_1$ and $N(t) \ge S^*(t)-\eps_1$. By conditions~P\ref{Cond-P1}), P\ref{Cond-P4}) and P\ref{Cond-P5}) there is a function $\psi$ such that $\psi(\xi) \to 0$ as $\xi \to 0$ and
\[
\begin{split}
\varphi(S(t),N(t),I(t))
& \le \varphi(S^*(t)+\eps_1,S^*(t)-\eps_1,I(t))\\
& = \dfrac{\varphi(S^*(t)+\eps_1,S^*(t)-\eps_1,I(t))}{I(t)}I(t)\\
& \le I(t) \lim_{\delta \to 0^+} \dfrac{\varphi(S^*(t)+\eps_1,S^*(t)-\eps_1,\delta)}{\delta}\\
& = \dfrac{\partial \varphi}{\partial I}(S^*(t)+\eps_1,S^*(t)-\eps_1,0) \ I(t)\\
& \le \left( \dfrac{\partial \varphi}{\partial I}(S^*(t),S^*(t),0)+\psi(\eps_1) \right) I(t),
\end{split}
\]
for $t>T_1$. Therefore, by the second and third equations in~\eqref{eq:ProblemaPrincipal}, we have
\[
\begin{cases}
E' \le \beta(t)\left[\dfrac{\partial \varphi}{\partial I}(S^*(t),S^*(t),0)I+\psi(\eps_1)I\right] -(\mu(t)+\eps(t))E\\
I' = \eps(t)E-(\mu(t)+\gamma(t))I
\end{cases}.
\]
Let
\[
M_2(t)=
\left[
\begin{array}{cc}
0 & \beta(t) \\
0 & 0
\end{array}
\right].
\]
By Theorem~\ref{local-stability-periodic-seirs} we conclude that $\rho(\Phi_{F-V}(\omega))<1$. Choose $\eps_1>0$ such that $\rho(\Phi_{F-V+\psi(\eps_1)M_2}(\omega))<1$ and consider the system
\[
\begin{cases}
u' = \beta(t)\left[\dfrac{\partial \varphi}{\partial I}(S^*(t),S^*(t),0)v+\psi(\eps_1)v\right] -(\mu(t)+\eps(t))u\\
v' = \eps(t)u-(\mu(t)+\gamma(t))v
\end{cases},
\]
or, in matrix language,
\[
\left[
\begin{array}{c}
u'\\
v'
\end{array}
\right]
=
\left(F(t)+V(t)
+\psi(\eps_1)M_2(t)
\right)
\left[
\begin{array}{c}
u\\
v
\end{array}
\right].
\]
By Lemma~\ref{lemma:global-stab-dfe-aux} and the standard comparison principle, there are $\omega$-periodic functions $v_1$ and $v_2$ such that $$E(t)\le v_1(t)\e^{p t} \quad \text{and} \quad I(t)\le v_2(t)\e^{p t},$$ where $p=\frac{1}{\omega} \ln(\rho(\Phi_{F-V+\psi(\eps_1)M_2}(\omega)))$. We conclude that $I(t)\to 0$ and $E(t)\to 0$ as $t \to +\infty$. It follows that $R(t)\to 0$ as $t \to +\infty$. Thus, since $N(t)-S^*(t) \to 0$ as $t \to +\infty$ we conclude that
$$S(t)-S^*(t)=N(t)-S^*(t)-E(t)-I(t)-R(t) \to 0,$$
as $t \to +\infty$. Hence the disease-free periodic solution is globally asymptotically stable. The result follows.
\end{proof}

\section{Persistence of the infective compartment and existence of endemic periodic orbits}\label{section:PersistenceExistence}

The next theorem shows that, when $\cR_0>1$, the infectives are persistent. Its proof consists in adapting the argument used in the first example in section 3 of~\cite{Rebelo-Margheri-Bacaer-JMB-2012}, where the case of a SEIRS model with simple incidence is considered, to our more general situation.

\begin{theorem}\label{teo:persist-infectives}
Assume that conditions P\ref{Cond-P0}) to P\ref{Cond-P5}) hold and that $\cR_0>1$. Then system~\eqref{eq:ProblemaPrincipal} is persistent with respect to $I$.
\end{theorem}

\begin{proof}
To prove the theorem we will use Theorem~3 in~\cite{Rebelo-Margheri-Bacaer-JMB-2012}. It follows from Lemma~\ref{lemma:invariant-region-periodic} that condition (A8) in Theorem 3 in~\cite{Rebelo-Margheri-Bacaer-JMB-2012} holds, letting the compact set $K$ be the set
$$K=\{(S,E,I,R)\in (\R^+_0)^4:\Lambda^\ell/\mu^u \le S+E+I+R \le \Lambda^u/\mu^\ell\}$$
if $\Lambda$ or $\mu$ are not constant functions and
$$K=\{(S,E,I,R)\in (\R^+_0)^4:\Lambda/\mu -\delta \le S+E+I+R \le \Lambda/\mu+\delta \},$$
for some $0<\delta<\Lambda/\mu$, if $\Lambda$ and $\mu$ are constant functions.

Let $(S^*(t),0,0,0)$ be the disease free periodic solution of system~\eqref{eq:ProblemaPrincipal}. If there is $\delta>0$ and $t_0 \in \R$ such that $I(t)\le\delta$ for $t\ge t_0$ then, using P\ref{Cond-P2}) and P\ref{Cond-P3}), we have
$$R'\le \gamma^u\delta-(\mu+\eta)^\ell R,$$
$$(S-S^*)'\le -\beta(t)\varphi(S,N,I)-\mu(t)(S-S^*)+\eta^u R\le -\mu^\ell(S-S^*)+\eta^u R,$$
$$E' \le \beta^u\varphi(S,N,I)-(\mu+\eps)^\ell E\le \beta^uc_2S\delta-(\mu+\eps)^\ell E$$
and
$$(S^*-S)'\le \beta(t)\varphi(S,N,I)-\mu(t)(S^*-S)-\eta^u R\le \beta^uc_2S\delta-\mu^\ell(S^*-S)$$
Thus, for $t$ sufficiently large, we have
\[
R(t) \le 2\delta \dfrac{\gamma^u}{(\mu+\gamma)^\ell} :=k_1(\delta),
\]
\[
S(t)-S^*(t) \le 2k_1(\delta)\dfrac{\eta^u}{\mu^\ell}:=k_2(\delta),
\]
\begin{equation}\label{eq:bound-persistence-E}
E(t) \le 2\delta\dfrac{c_2\beta^u(k_2(\delta)+S^*)^u}{(\mu+\eps)^\ell}
:=k_3(\delta)
\end{equation}
and
\begin{equation}\label{eq:bound-persistence-S-S^*}
S^*(t)-S(t) \le 2\delta\dfrac{c_2\beta^u(k_2(\delta)+S^*)^u}{\mu^\ell}
:=k_4(\delta).
\end{equation}
Also, according to~\eqref{eq:bound-N}, we also have, for $t$ sufficiently large,
\begin{equation}\label{eq:bound-persistence-N}
|S^*(t)-N(t)| \le k_5(\delta),
\end{equation}
with $k_5(\delta)\to 0$ as $\delta \to 0$.

Now, we will check assumptions (ii) and (iii) (a) in Theorem~3 in~\cite{Rebelo-Margheri-Bacaer-JMB-2012}.
Assume that there exists $t_0 \in \R$ such that $I(t)\le\delta$ for each $t\ge t_0$. From~\eqref{eq:bound-persistence-E}, there exists $t_3\ge t_0$ such that for each $t\ge t_3$ we have $E(t)\le k_3(\delta)$. So we obtain~(iii) (a) in Theorem~3 in~\cite{Rebelo-Margheri-Bacaer-JMB-2012} setting $\eta(\delta)=k_3(\delta)$ and (i) holds since $\eta(\delta)\to 0$ as $\delta \to 0$. Let us now check assumptions (i) and (iii) (b) in Theorem~3 in~\cite{Rebelo-Margheri-Bacaer-JMB-2012}. Choose $\delta_1>0$ such that $k_4(\delta) < \displaystyle \min_{t \in [0,\omega)} S^*(t)$ for all $0<\delta<\delta_1$. Take $\delta \in (0,\delta_1)$ and suppose that there exists
$t_0\in\R$ such that $\|(E(t),I(t))\|\le \delta$ for each $t \ge t_0$. Then~\eqref{eq:bound-persistence-S-S^*} shows that there exists
$t_4\ge t_0$ such that $S(t)\ge S^*(t)-k_4(\delta)$ for $t \ge t_4$ and~\eqref{eq:bound-persistence-N} shows that $N(t)\le S^*(t)+k_5(\delta)$. Therefore, by P\ref{Cond-P4}), we get
\[
\begin{cases}
E' \ge \beta(t)\varphi(S^*(t)-k_4(\delta),S^*(t)+k_5(\delta),I)-(\mu(t)+\eps(t)) E\\
I' \ge \eps(t)E-(\mu(t)+\gamma(t))I
\end{cases}
\]
and assumption (iii) (b) in Theorem~3 in~\cite{Rebelo-Margheri-Bacaer-JMB-2012} holds with
\[
\lambda(\delta)=\max_{t \in [0,\omega]} \,
\dfrac{\partial \varphi/\partial I \, (S^*(t),S^*(t),0)} {\varphi(S^*(t)-k_4(\delta),S^*(t)+k_5(\delta),\delta)/\delta}.
\]
Since $\lambda(\delta)\to 1$ as $\delta \to 0$ we conclude that (ii) in the referred theorem holds. The result follows.
\end{proof}

We need the following auxiliary result that will be used to show the existence and uniqueness of the solution of some algebraic equations in the proof of our main result.
\begin{lemma}\label{lemma:unique-sol}
Assume that condition P\ref{Cond-P0}) to P\ref{Cond-P4}) hold. Then there is a unique $r>0$ that solves equation
\begin{equation}\label{eq:sol-e^u_3-lemma}
\dfrac{\bar\epsilon\bar\beta}{\bar\mu+\bar\gamma} \, \varphi\left(\bar\Lambda/\bar\mu - d r, \bar\Lambda/\bar\mu, r\right)/r-(\bar\mu+\bar\epsilon)=0,
\end{equation}
where
\[
d=\frac{(\bar\mu+\bar\gamma)(\bar\mu+\bar\epsilon)(\bar\mu+\bar\eta)-\bar\epsilon\bar\gamma\bar\eta}{\bar\epsilon\bar\mu(\bar\mu+\bar\eta)}.
\]
This unique solution belongs to the interval $]0,\bar\Lambda/\bar\mu[$. \end{lemma}

\begin{proof}
According to conditions P\ref{Cond-P1}), P\ref{Cond-P2}) and P\ref{Cond-P5}), the function $\psi:[0,\bar\Lambda/\bar\mu] \to \R$ given by
\[
\psi(v)=
\begin{cases}
\dfrac{\bar\epsilon\bar\beta}{\bar\mu+\bar\gamma} \dfrac{\varphi\left(\bar\Lambda/\bar\mu - d v, \bar\Lambda/\bar\mu, v\right)}{v} -(\bar\mu+\bar\epsilon) & \quad \text{if} \ \ 0<v\le\bar\Lambda/\bar\mu\\[2mm]
\dfrac{\bar\epsilon\bar\beta}{\bar\mu+\bar\gamma} \, \dfrac{\partial \varphi}{\partial I}\left(\bar\Lambda/\bar\mu, \bar\Lambda/\bar\mu, 0\right) -(\bar\mu+\bar\epsilon) & \quad \text{if} \ \ v = 0
\end{cases}
\]
is continuous and non-increasing and we have
    $$\psi(0)=\left[\dfrac{\bar\epsilon\bar\beta}{(\bar\mu+\bar\gamma)(\bar\mu+\bar\epsilon)} \displaystyle \dfrac{\partial \varphi}{\partial I}\left(\bar\Lambda/\bar\mu, \bar\Lambda/\bar\mu, 0\right) - 1\right](\bar\mu+\bar\epsilon)=\left(\overline{R}_0-1\right)(\bar\mu+\bar\epsilon)>0.$$
By P\ref{Cond-P2}), for the unique $d_0 \in ]0,\bar\Lambda/\bar\mu[$ satisfying $\bar\Lambda/\bar\mu-d d_0=0$, we get
    $$\psi\left(d_0\right)=\left[\dfrac{\bar\epsilon\bar\beta}{(\bar\mu+\bar\gamma)(\bar\mu+\bar\epsilon)} \displaystyle \dfrac{\varphi\left(0, \bar\Lambda/\bar\mu, d_0 \right)}{d_0} -1\right](\bar\mu+\bar\epsilon)=-(\bar\mu+\bar\epsilon)<0.$$
Thus, by Bolzano's theorem, there is $r \in ]0,d_0[ \subset ]0,\bar\Lambda/\bar\mu[$ that solves~\eqref{eq:sol-e^u_3-lemma}. Since
\[
\psi'(v)=\dfrac{\bar\eps\bar\beta}{\bar\mu+\bar\gamma} \
\dfrac{\left[-d\frac{\partial\varphi}{\partial S}(c(v))+\frac{\partial\varphi}{\partial I}(c(v))\right]v-
\varphi(c(v))}{v^2}<0,
\]
where $c(v)=(\bar\Lambda/\bar\mu-dv,\bar\Lambda/\bar\mu,v)$ (note that, by P\ref{Cond-P5}) we have $\frac{\partial\varphi}{\partial I}(c(v))v-\varphi(c(v))<0$ and by~P\ref{Cond-P4}) we have $\frac{\partial\varphi}{\partial S}(c(v))\ge 0$), we conclude that the solution is unique and the proof is complete.
\end{proof}

We also need to consider the matrix
\begin{equation}\label{eq:matrixM}
\mathcal M=
\left[
\begin{array}{cccc}
-\mu-K_{110} &
-K_{010}q/p &
-K_{011}r/p &
\left(-K_{010}+\eta\right)s/p \\
K_{110}p/q &
K_{010} &
K_{011}r/q &
K_{010}s/q \\
0 & \mu+\gamma & -(\mu+\gamma) & 0 \\
0 & 0 & \mu+\eta & -(\mu+\eta)
\end{array}
\right]
\end{equation}
where $r$ is the unique solution of~\eqref{eq:sol-e^u_3-lemma},
$$p=\frac{\bar\Lambda}{\bar\mu}- \frac{(\bar\mu+\bar\gamma)(\bar\mu+\bar\epsilon)(\bar\mu+\bar\eta) -\bar\epsilon\bar\gamma\bar\eta}{\bar\epsilon\bar\mu(\bar\mu+\bar\eta)}r, \quad q=(\bar\mu+\bar\gamma)r/\bar\eps, \quad s=\bar\gamma r/(\bar\mu+\bar\eta)$$
and
$$K_{abc}=\bar\beta\left[a\frac{\partial \varphi}{\partial S}(p,\bar\Lambda/\bar\mu,r)+b\frac{\partial \varphi}{\partial N}(p,\bar\Lambda/\bar\mu,r)+c\frac{\partial \varphi}{\partial I}(p,\bar\Lambda/\bar\mu,r)\right].$$

In the following result, we obtain conditions for the existence of endemic periodic orbits.

\begin{theorem}\label{teo:main}
  Assume that conditions~P\ref{Cond-P1}) to P\ref{Cond-P5}) hold. Assume also that
  \begin{enumerate}[1)]
    \item \label{cond-C0} $\mathcal R_0> 1$;
    \item \label{cond-C3} $\det \mathcal M \ne 0$.
  \end{enumerate}
  Then system~\eqref{eq:ProblemaPrincipal} has an endemic $\omega$-periodic solution.
\end{theorem}

To obtain Theorem~\ref{teo:main} we will use a well known result in degree theory, the Mawhin continuation theorem~\cite{Gaines-Mawhin-NDE-1977,Mawhin-JDE-1972}.

\begin{proof}
Before proving Theorem~\ref{teo:main}, we first need to give some definitions and state some well known facts. Let $X$ and $Z$ be Banach spaces.
\begin{definition}
A linear mapping $\mathcal L: D \subseteq X \to Z$ is called a \emph{Fredholm mapping of index zero} if
\begin{enumerate}
\item $\dim \ker \mathcal L = \codim \imagem \mathcal L < \infty$;
\item $\imagem \mathcal L$ is closed in $Z$.
\end{enumerate}
\end{definition}
Given a Fredholm mapping of index zero, $\mathcal L: D \subseteq X \to Z$ , it is well known that there are continuous projectors $P:X\to X$ and $Q:Z\to Z$ such that
\begin{enumerate}
  \item $\imagem P = \ker \mathcal L$;
  \item $\ker Q = \imagem \mathcal L = \imagem (I-Q)$;
  \item $X = \ker \mathcal L \oplus \ker P$;
  \item $Z = \imagem \mathcal L \oplus \imagem Q$.
\end{enumerate}
It follows that $\mathcal L|_{D \cap \ker P}: (I-P)X \to \imagem \mathcal L$ is invertible. We denote the inverse of that map by $K_p$.
\begin{definition}
A continuous mapping $\mathcal N: X \to Z$ is called $L$-compact on $\overline{U} \subset X$, where $U$ is an open bounded set, if
\begin{enumerate}
\item $Q\mathcal N(\overline U)$ is bounded;
\item $K_p(I-Q)\mathcal N: \overline{U} \to X$ is compact.
\end{enumerate}
\end{definition}
Since $\imagem Q$ is isomorphic to $\ker \mathcal L$, there exists an isomorphism $\mathcal J: \imagem Q \to \ker \mathcal L$.

We are now prepared to state the theorem that will allow us to prove Theorem~\ref{teo:main}: Mawhin's continuation theorem~\cite{Mawhin-JDE-1972}.

\begin{theorem}(Mawhin's continuation theorem)
  Let $X$ and $Z$ be Banach spaces, let $U \subset X$ be an open and bounded set, let $\mathcal L: D \subseteq X \to Z$ be a  Fredholm mapping of index zero and let $\mathcal N: X \to Z$ be $L$-compact on $\overline{U}$. Assume that
\begin{enumerate}[1)]
  \item for each $\lambda \in (0,1)$ and $x \in \partial U \cap D$ we have $\mathcal L x \ne \lambda \mathcal N x$;
  \item for each  $x \in \partial U \cap \ker \mathcal L$ we have $Q\mathcal N x \ne 0$;
  \item $\deg(\mathcal J Q \mathcal N, U \cap \ker \mathcal L, 0) \ne 0$.
\end{enumerate}
Then the operator equation $\mathcal L x= \mathcal N x$ has at least one solution in $D \cap \overline U$.
\end{theorem}

With the change of variables $S(t)=\e^{u_1(t)}$, $E(t)=\e^{u_2(t)}$, $I(t)=\e^{u_3(t)}$ and $R(t)=\e^{u_4(t)}$, system~\eqref{eq:ProblemaPrincipal} becomes

\begin{equation}\label{eq:sistema-aplic-Mawhin}
\begin{cases}
u_1'=\Lambda(t)\e^{-u_1}-\beta(t)\,\varphi(\e^{u_1},w,\e^{u_3})\e^{-u_1} -\mu(t)+\eta(t)\e^{u_4-u_1}\\
u_2'=\beta(t)\,\varphi(\e^{u_1},w,\e^{u_3})\e^{-u_2}-(\mu(t)+\epsilon(t))\\
u_3'=\epsilon(t)\e^{u_2-u_3} -(\mu(t)+\gamma(t)) \\
u_4'=\gamma(t)\e^{u_3-u_4}-(\mu(t)+\eta(t)) \\
w=\e^{u_1}+\e^{u_2}+\e^{u_3}+\e^{u_4}
\end{cases}
\end{equation}
and if $(v_1(t),v_2(t),v_3(t),v_4(t))$ is a periodic solution of period $\omega$ of system~\eqref{eq:sistema-aplic-Mawhin} then
$\left(\e^{v_1(t)},\e^{v_2(t)},\e^{v_3(t)},\e^{v_4(t)}\right)$ is a periodic solution of period $\omega$ of system~\eqref{eq:ProblemaPrincipal}.
Consider the system
\begin{equation}\label{eq:sistema-aplic-Mawhin-lambda}
\begin{cases}
u_1'=\lambda \left( \Lambda(t)\e^{-u_1}-\beta(t) \varphi(\e^{u_1},w,\e^{u_3}) \e^{-u_1} -\mu(t)+\eta(t)\e^{u_4-u_1} \right)\\
u_2'=\lambda \left( \beta(t) \varphi(\e^{u_1},w,\e^{u_3}) \e^{-u_2} -(\mu(t)+\epsilon(t)) \right)\\
u_3'=\lambda \left( \epsilon(t)\e^{u_2-u_3} -(\mu(t)+\gamma(t)) \right)\\
u_4'=\lambda \left( \gamma(t)\e^{u_3-u_4}-(\mu(t)+\eta(t)) \right)\\
w=\e^{u_1}+\e^{u_2}+\e^{u_3}+\e^{u_4}
\end{cases}.
\end{equation}
By~\ref{lemma:exist-per-sol-DF-1}) in Lemma~\ref{lemma:exist-per-sol-DF}, if $(u_1(t),u_2(t),u_3(t),u_4(t))$ is periodic then
\begin{equation}\label{eq:bounds_w}
  \frac{\Lambda^\ell}{\mu^u} \le w(t) \le \frac{\Lambda^u}{\mu^\ell}.
\end{equation}

We will now prepare the setting where we will apply Mawhin's theorem. We will
consider the Banach spaces $(X,\|\cdot\|)$ and $(Z,\|\cdot\|)$ where
$$X=Z=\{u = (u_1,u_2,u_3,u_4) \in C(\R,\R^4): u(t)=u(t+\omega) \}$$
and
$$\|u\|=\max_{t \in [0,\omega]} |u_1(t)|+\max_{t \in [0,\omega]} |u_2(t)|+\max_{t \in [0,\omega]} |u_3(t)|+\max_{t \in [0,\omega]} |u_4(t)|.$$
Let $\mathcal L: D \subseteq X \to Z$, where $D=X \cap C^1(\R,\R^4)$, be defined by
\[
\mathcal L u(t) = \dfrac{d u(t)}{dt}
\]
and $\mathcal N: X \to Z$ be defined by
\[
\mathcal N u(t) =
\left[
\begin{array}{c}
\Lambda(t)\e^{-u_1(t)}-\beta(t)
\varphi(\e^{u_1},w,\e^{u_3})
\e^{-u_1(t)} -\mu(t)+\eta(t)\e^{u_4(t)-u_1(t)}\\[2mm]
\beta(t) \varphi(\e^{u_1},w,\e^{u_3}) \e^{-u_2(t)}
-(\mu(t)+\epsilon(t))\\[2mm]
\epsilon(t)\e^{u_2(t)-u_3(t)} -(\mu(t)+\gamma(t))\\[2mm]
\gamma(t)\e^{u_3(t)-u_4(t)}-(\mu(t)+\eta(t))
\end{array}
\right].
\]
Consider also the projectors $P:X\to X$ and $Q:Z\to Z$ given by
$$ Pu = \dfrac{1}{\omega} \int_0^\omega u(t) \dt \quad \quad \text{and} \quad \quad
Qz = \dfrac{1}{\omega} \int_0^\omega z(t) \dt.$$
Note that $\imagem P = \ker \mathcal L = \R^4$, that
$$\ker Q = \imagem \mathcal L = \imagem (I-Q)
= \left\{z \in Z: \dfrac{1}{\omega} \int_0^\omega z(t) \dt =0\right\},$$
that $\mathcal L$ is a Fredholm mapping of index zero (since
$\dim \ker \mathcal L =\codim \imagem \mathcal L =4$) and that $\imagem \mathcal L$
is closed in $X$.

Consider the generalized inverse of $\mathcal L$,
$\mathcal K_p: \imagem \mathcal L \to D \cap \ker P$, given by
$$\mathcal K_p z(t)=\int_0^t z(s)\ds - \dfrac{1}{\omega} \int_0^\omega \int_0^r z(s) \ds \dr,$$
$t \in [0,\omega]$, the operator $Q \mathcal N:X \to Z$ given by
\[
Q \mathcal N u(t) =
\left[
\begin{array}{c}
\dfrac{1}{\omega} \dint_0^\omega \frac{\Lambda(t)}{\e^{u_1(t)}} - \beta(t)
\varphi(\e^{u_1},w,\e^{u_3})\e^{-u_1(t)} + \frac{\eta(t)\e^{u_4(t)}}{\e^{u_1(t)}} \dt  -\bar \mu\\[2mm]
\dfrac{1}{\omega} \dint_0^\omega  \beta(t) \varphi(\e^{u_1},w,\e^{u_3}) \e^{-u_2(t)} \dt
-(\bar \mu+\bar \epsilon)\\[2mm]
\dfrac{1}{\omega} \dint_0^\omega \epsilon(t)\e^{u_2(t)-u_3(t)} \dt - (\bar\mu+\bar\gamma)\\[2mm]
\dfrac{1}{\omega} \dint_0^\omega \gamma(t)\e^{u_3(t)-u_4(t)} \dt -(\bar\mu+\bar\eta)
\end{array}
\right].
\]
and the mapping $\mathcal K_p (I-Q)\mathcal N:X \to D \cap \ker P$ given by
    $$\mathcal K_p (I-Q) \mathcal N u(t) = A_1(t)-A_2(t)-A_3(t)$$
where
\[
A_1(t)=
\left[
\begin{array}{c}
\dint_0^t \frac{\Lambda(t)}{\e^{u_1(t)}} - \beta(t)
\varphi(\e^{u_1},w,\e^{u_3})
\e^{-u_1(t)} + \frac{\eta(t)\e^{u_4(t)}}{\e^{u_1(t)}} - \mu(t) \dt \\[2mm]
\dint_0^t \beta(t) \varphi(\e^{u_1},w,\e^{u_3})  \e^{-u_2(t)}
-( \mu(t)+ \epsilon(t)) \dt
\\[2mm]
\dint_0^t \epsilon(t)\e^{u_2(t)-u_3(t)} - (\mu(t)+\gamma(t)) \dt \\[2mm]
\dint_0^t  \gamma(t)\e^{u_3(t)-u_4(t)} -(\mu(t)+\eta(t)) \dt
\end{array}
\right],
\]
\[
A_2(t)=
\left[
\begin{array}{c}
\dfrac{1}{\omega} \dint_0^\omega \dint_0^t \frac{\Lambda(s)}{\e^{u_1(s)}} - \beta(s)
\varphi(\e^{u_1},w,\e^{u_3}) \e^{-u_1(s)} + \frac{\eta(s)\e^{u_4(s)}}{\e^{u_1(s)}} - \mu(s) \ds \dt \\[2mm]
\dfrac{1}{\omega} \dint_0^\omega \dint_0^t \beta(s) \varphi(\e^{u_1},w,\e^{u_3})  \e^{-u_2(s)}
-( \mu(s)+ \epsilon(s)) \ds \dt
\\[2mm]
\dfrac{1}{\omega} \dint_0^\omega \dint_0^t \epsilon(t)\e^{u_2(s)-u_3(s)} - (\mu(s)+\gamma(s)) \ds \dt \\[2mm]
\dfrac{1}{\omega} \dint_0^\omega \dint_0^t  \gamma(t)\e^{u_3(s)-u_4(s)} -(\mu(s)+\eta(s)) \ds \dt
\end{array}
\right]
\]
and
\[
A_3(t)=\left[\frac{t}{\omega}-\frac12\right]
\left[
\begin{array}{c}
\dint_0^\omega \frac{\Lambda(t)}{\e^{u_1(t)}} - \beta(t)
\varphi(\e^{u_1},w,\e^{u_3})
\e^{-u_1(t)} + \frac{\eta(t)\e^{u_4(t)}}{\e^{u_1(t)}} - \mu(t) \dt \\[2mm]
\dint_0^\omega \beta(t) \varphi(\e^{u_1},w,\e^{u_3})  \e^{-u_2(t)}
-( \mu(t)+ \epsilon(t)) \dt
\\[2mm]
\dint_0^\omega \epsilon(t)\e^{u_2(t)-u_3(t)} - (\mu(t)+\gamma(t)) \dt \\[2mm]
\dint_0^\omega  \gamma(t)\e^{u_3(t)-u_4(t)} -(\mu(t)+\eta(t)) \dt
\end{array}
\right].
\]
It is immediate that $Q\mathcal N$ and $\mathcal K_p (I-Q) \mathcal N$ are continuous.
 An application of Ascoli-Arzela's theorem shows that $\mathcal K_p (I-Q) \mathcal N (\overline\Omega)$
is compact for any bounded set $\Omega \subset X$. Since $Q\mathcal N (\overline\Omega)$ is bounded, we conclude that $\mathcal N$
is $L$-compact on $\Omega$ for any bounded set $\Omega \subset X$.

Let $(u_1,u_2,u_3,u_4) \in X$ be some solution of~\eqref{eq:sistema-aplic-Mawhin-lambda} for some $\lambda \in (0,1)$ and, for $i=1,2,3,4$ define
    $$u_i(\xi_i) = \min_{t \in [0,\omega]} u_i(t) \quad \quad \text{and} \quad \quad u_i(\chi_i) = \max_{t \in [0,\omega]} u_i(t).$$
From the third equation in~\eqref{eq:sistema-aplic-Mawhin-lambda} we get,
\begin{equation}\label{eq:u2-u3_xi}
\e^{u_2(\xi_2)-u_3(\xi_3)} \le \e^{u_2(\xi_3)-u_3(\xi_3)} = \dfrac{\mu(\xi_3)+\gamma(\xi_3)}{\epsilon(\xi_3)} \le \dfrac{(\mu+\gamma)^u}{\epsilon^\ell}
\end{equation}
and
\begin{equation}\label{eq:u2-u3_eta}
\e^{u_2(\chi_2)-u_3(\chi_3)} \ge \e^{u_2(\chi_3)-u_3(\chi_3)} = \dfrac{\mu(\chi_3)+\gamma(\chi_3)}{\epsilon(\chi_3)} \ge \dfrac{(\mu+\gamma)^\ell}{\epsilon^u}.
\end{equation}
From the second equation in~\eqref{eq:sistema-aplic-Mawhin-lambda}, P\ref{Cond-P3}) and~\eqref{eq:u2-u3_xi}, we obtain
\[
\begin{split}
\e^{u_1(\xi_1)} \le \e^{u_1(\xi_2)}
& = \dfrac{(\mu+\epsilon)^u}{\beta^\ell} \dfrac{\e^{u_1(\xi_2)+u_3(\xi_2)}}{\varphi(\e^{u_1(\xi_2)},w(\xi_2),\e^{u_3(\xi_2)})} \e^{u_2(\xi_2)-u_3(\xi_2)} \\
& \le \dfrac{(\mu+\epsilon)^u}{\beta^\ell} \dfrac{\e^{u_1(\xi_2)+u_3(\xi_2)}}{\varphi(\e^{u_1(\xi_2)},w(\xi_2),\e^{u_3(\xi_2)})} \dfrac{(\mu+\gamma)^u}{\epsilon^\ell} \\
& \le \dfrac{(\mu+\epsilon)^u(\mu+\gamma)^u}{c_1\beta^\ell\epsilon^\ell}
\end{split}
\]
and, by the second equation in~\eqref{eq:sistema-aplic-Mawhin-lambda}, P\ref{Cond-P3}) and~\eqref{eq:u2-u3_eta}, we get
\begin{equation}\label{eq:min_u1_eta1}
\begin{split}
\e^{u_1(\chi_1)} \ge \e^{u_1(\chi_2)}
& = \dfrac{(\mu+\epsilon)^\ell}{\beta^u} \dfrac{\e^{u_1(\chi_2)+u_3(\chi_2)}}{\varphi(\e^{u_1(\chi_2)},w(\chi_2),\e^{u_3(\chi_2)})} \e^{u_2(\chi_2)-u_3(\chi_2)} \\
& \ge \dfrac{(\mu+\epsilon)^\ell}{\beta^u} \dfrac{\e^{u_1(\xi_2)+u_3(\xi_2)}}{\varphi(\e^{u_1(\xi_2)},w(\xi_2),\e^{u_3(\xi_2)})} \dfrac{(\mu+\gamma)^\ell}{\epsilon^u} \\
& \ge \dfrac{(\mu+\epsilon)^\ell(\mu+\gamma)^\ell}{c_2\beta^u\epsilon^u}.
\end{split}
\end{equation}
Define
\begin{equation}\label{eq:starA1}
A_{1\xi}=\dfrac{(\mu+\epsilon)^u(\mu+\gamma)^u}{c_1\beta^\ell\epsilon^\ell} \quad \text{and} \quad A_{1\chi}=\dfrac{(\mu+\epsilon)^\ell(\mu+\gamma)^\ell}{c_2\beta^u\epsilon^u}.
\end{equation}
From the fourth equation in~\eqref{eq:sistema-aplic-Mawhin-lambda} we get
\[
\e^{u_3(\xi_3)} \le \e^{u_3(\chi_4)-u_4(\chi_4)+u_4(\chi_4)}
=\frac{\mu(\chi_4)+\eta(\chi_4)}{\gamma(\chi_4)} \e^{u_4(\chi_4)} \le \frac{(\mu+\eta)^u}{\gamma^\ell} \e^{u_4(\chi_4)}
\]
and
\[
\e^{u_3(\chi_3)} \ge \e^{u_3(\xi_4)-u_4(\xi_4)}\e^{u_4(\xi_4)} = \frac{\mu(\xi_4)
+\eta(\xi_4)}{\gamma(\xi_4)}\e^{u_4(\xi_4)} \ge \frac{(\mu+\eta)^\ell}{\gamma^u} \e^{u_4(\xi_4)}.
\]
Thus we obtain
\begin{equation}\label{eq:maj_u4-xi4}
\e^{u_4(\xi_4)} \le \frac{\gamma^u}{(\mu+\eta)^\ell} \e^{u_3(\chi_3)}
\quad\quad \text{and} \quad \quad
\e^{u_4(\chi_4)} \ge \frac{\gamma^\ell}{(\mu+\eta)^u} \e^{u_3(\xi_3)}.
\end{equation}

From the first equation in~\eqref{eq:sistema-aplic-Mawhin-lambda} we have
$$\beta(\chi_1) \varphi\left(\e^{(u_1(\chi_1))},w(\chi_1),\e^{u_3(\chi_1)}\right)
= \Lambda(\chi_1) -\mu(\chi_1)\e^{u_1(\chi_1)} + \eta(\chi_1)\e^{u_4(\chi_1)}.$$
Using~\eqref{eq:min_u1_eta1} and~\eqref{eq:bounds_w}, the right hand expression can be bounded by
\begin{equation}\label{eq:boundS1}
\begin{split}
\Lambda(\chi_1) -\mu(\chi_1)\e^{u_1(\chi_1)} + \eta(\chi_1)\e^{u_4(\chi_1)}
& \le \Lambda^u -\mu^\ell\e^{u_1(\chi_1)} + \eta^u \e^{u_4(\chi_1)}\\
& \le \Lambda^u + \eta^u \frac{\Lambda^u}{\mu^\ell}
\end{split}
\end{equation}
and, by \eqref{eq:min_u1_eta1}, we obtain
\begin{equation}\label{eq:boundS2}
\begin{split}
\beta(\chi_1) \varphi\left(\e^{u_1(\chi_1)},w(\chi_1),\e^{u_3(\chi_1)}\right)
& \ge \beta^\ell c_1
\e^{u_1(\chi_1)+u_3(\chi_1)} \\
& \ge  \dfrac{\beta^\ell c_1 (\mu+\epsilon)^\ell(\mu+\gamma)^\ell}{c_2\beta^u\epsilon^u} \e^{u_3(\xi_3)}.
\end{split}
\end{equation}
By~\eqref{eq:boundS1}~and~\eqref{eq:boundS2} we get
\begin{equation}\label{eq:exp_u3_xi3}
\e^{u_3(\xi_3)}
\le \dfrac{c_2(1 + \eta^u/\mu^\ell)\Lambda^u\beta^u\epsilon^u}{c_1\beta^\ell (\mu+\epsilon)^\ell(\mu+\gamma)^\ell}.
\end{equation}
By hypothesis~\ref{cond-C0}), we have $\mathcal R_0>1$ and thus, by Theorem~\ref{teo:persist-infectives}, there is $K^\ell>0$ such that
\begin{equation}\label{eq:bound-I-R0>1}
\liminf_{t \to +\infty} I(t) \ge K^\ell.
\end{equation}
Thus $\e^{u_3(t)}\ge K^\ell$. Define
\begin{equation}\label{eq:starA3-1}
  A_{3\xi}=\dfrac{c_2(1 + \eta^u/\mu^\ell)\Lambda^u\beta^u\epsilon^u}{c_1\beta^\ell (\mu+\epsilon)^\ell(\mu+\gamma)^\ell}
  \quad \text{and} \quad A_{3\chi}=K^\ell.
\end{equation}

Using~\eqref{eq:bound-I-R0>1},~\eqref{eq:bounds_w} and~\eqref{eq:maj_u4-xi4} and again the fact that $\cR_0>1$, we obtain bounds for $\e^{u_4(t)}$, namely
$$\e^{u_4(\xi_4)} \le \frac{\gamma^u}{(\mu+\eta)^\ell} \dfrac{\Lambda^u}{\mu^\ell}
\quad\quad \text{and} \quad \quad
\e^{u_4(\chi_4)} \ge \frac{\gamma^\ell}{(\mu+\eta)^u} \e^{u_3(\xi_3)} \ge \frac{\gamma^\ell}{(\mu+\eta)^u} K^\ell.$$
Define
\begin{equation}\label{eq:starA4}
A_{4\xi} = \frac{\gamma^u}{(\mu+\eta)^\ell} \dfrac{\Lambda^u}{\mu^\ell} \quad \text{and} \quad
A_{4\chi} = \frac{\gamma^\ell}{(\mu+\eta)^u} K^\ell.
\end{equation}

By the third equation in~\eqref{eq:ProblemaPrincipal},~\eqref{eq:exp_u3_xi3} and~\eqref{eq:bound-I-R0>1} we get
$$\e^{u_2(\xi_2)} \le \e^{u_2(\xi_3)-u_3(\xi_3)} \e^{u_3(\xi_3)} \le \dfrac{(\mu+\gamma)^u}{\epsilon^\ell} A_{3\xi}$$
and
$$\e^{u_2(\chi_2)} \ge \e^{u_2(\chi_3)-u_3(\chi_3)} \e^{u_3(\chi_3)} \ge \dfrac{(\mu+\gamma)^\ell}{\epsilon^u} A_{3\chi}.$$
Using~\eqref{eq:starA3-1}, we can establish bounds for $\e^{u_2(t)}$. In fact, we have $\e^{u_2(\xi_2)}\le A_{2\xi}$ and $\e^{u_2(\chi_2)}\ge A_{2\chi}$, where
\begin{equation}\label{eq:starA2-1}
A_{2\xi}= \dfrac{c_2(\mu+\gamma)^\ell(1 + \eta^u/\mu^\ell)\Lambda^u\beta^u\epsilon^u}{c_1\epsilon^\ell\beta^\ell (\mu+\epsilon)^\ell(\mu+\gamma)^\ell}
\end{equation}
and
\begin{equation}\label{eq:starA2-2}
A_{2\chi} = \dfrac{(\mu+\gamma)^\ell}{\epsilon^u}K^\ell.
\end{equation}

By~\eqref{eq:starA1}, \eqref{eq:starA3-1}, \eqref{eq:starA3-1}, \eqref{eq:starA4}, \eqref{eq:starA2-1}, \eqref{eq:starA2-2} we obtain, for $i=1,\ldots,4$,
\begin{equation}\label{eq:ui_le_ln(Ai)}
u_i(\xi_i) \le \ln A_{i\xi} \quad \quad \text{and} \quad \quad u_i(\chi_i) \ge \ln A_{i\chi}.
\end{equation}

Integrating in $[0,\omega]$ the last three equations in~\eqref{eq:sistema-aplic-Mawhin-lambda} we obtain
\begin{equation}\label{eq:qualquer-coisa-1}
\int_0^\omega \beta(t)\,\varphi\left(\e^{u_1(t)},w(t),\e^{u_3(t)}\right)\e^{-u_2(t)} \dt = (\bar \mu + \bar\epsilon)\omega,
\end{equation}
\begin{equation}\label{eq:qualquer-coisa-2}
\int_0^\omega \epsilon(t)\e^{u_2(t)-u_3(t)} \dt = (\bar\mu+\bar\gamma)\omega
\end{equation}
and
\begin{equation}\label{eq:qualquer-coisa-3}
\int_0^\omega \gamma(t)\e^{u_3(t)-u_4(t)} = (\bar\mu+\bar\eta)\omega.
\end{equation}

By~\eqref{eq:ui_le_ln(Ai)} and~\eqref{eq:qualquer-coisa-1} and using the fact that $\lambda\in(0,1)$, we get
\[
\begin{split}
u_2(t)
& = u_2(\xi_2) + \int_{\xi_2}^t u_2'(s) \ds \le u_2(\xi_2) + \int_0^\omega |u_2'(t)| \dt \\
& = u_2(\xi_2) + \lambda \int_0^\omega
\left| \beta(t)\,\varphi\left(\e^{u_1(t)},w(t),\e^{u_3(t)}\right)\e^{-u_2(t)}
-(\mu(t)+\epsilon(t))  \right| \dt \\
& \le \ln A_{2\xi} + 2 \int_0^\omega
\beta(t)\,\varphi\left(\e^{u_1(t)},w(t),\e^{u_3(t)}\right)\e^{-u_2(t)} \dt \\
& \le \ln A_{2\xi} +2(\bar \mu + \bar\epsilon)\omega,
\end{split}
\]
and also
\[
\begin{split}
u_2(t)
& \ge u_2(\chi_2) - \int_0^\omega |u_2'(t)| \dt \\
& = u_2(\chi_2) - \int_0^\omega
\left| \beta(t)\,\varphi\left(\e^{u_1(t)},w(t),\e^{u_3(t)}\right)\e^{-u_2(t)}
-(\mu(t)+\epsilon(t))  \right| \dt \\
& \ge \ln A_{2\chi} - 2(\bar \mu + \bar\epsilon)\omega.
\end{split}
\]
By~\eqref{eq:ui_le_ln(Ai)} and~\eqref{eq:qualquer-coisa-2} and using the fact that $\lambda\in(0,1)$,  we obtain
\begin{equation}\label{eq:qualquer-coisa-4}
\begin{split}
u_3(t)
& \le u_3(\xi_3) + \int_0^\omega |u_3'(t)| \dt
= u_3(\xi_3) + \lambda \int_0^\omega
\left| \epsilon(t)\e^{u_2-u_3} -(\mu(t)+\gamma(t)) \right| \dt \\
& \le \ln A_{3\xi} + 2 \int_0^\omega \epsilon(t)\e^{u_2-u_3} \dt
\le \ln A_{3\xi} +2(\bar \mu + \bar\gamma)\omega,
\end{split}
\end{equation}
and also
\[
\begin{split}
u_3(t)
& \ge u_3(\chi_3) - \int_0^\omega |u_3'(t)| \dt
= u_3(\chi_3) - \lambda \int_0^\omega
\left| \epsilon(t)\e^{u_2-u_3} -(\mu(t)+\gamma(t)) \right| \dt \\
& \ge \ln A_{3\chi} - 2 \int_0^\omega \epsilon(t)\e^{u_2-u_3} \dt
\ge \ln A_{3\chi} -2(\bar \mu + \bar\gamma)\omega.
\end{split}
\]
Similarly, by~\eqref{eq:ui_le_ln(Ai)} and~\eqref{eq:qualquer-coisa-3} and using the fact that $\lambda\in(0,1)$, we conclude that
\[
\begin{split}
u_4(t)
& \le u_4(\xi_4) + \int_0^\omega |u_4'(t)| \dt
 = u_4(\xi_4) + \lambda \int_0^\omega
\left| \gamma(t)\e^{u_3-u_4}-(\mu(t)+\eta(t)) \right| \dt \\
& \le \ln A_{4\xi} + 2 \int_0^\omega \gamma(t)\e^{u_3-u_4} \dt
\le \ln A_{4\xi} +2(\bar\mu+\bar\eta)\omega
\end{split}
\]
and also that
\[
\begin{split}
u_4(t)
& \ge u_4(\chi_4) - \int_0^\omega |u_4'(t)| \dt
= u_4(\chi_4) - \lambda \int_0^\omega
\left| \gamma(t)\e^{u_3-u_4}-(\mu(t)+\eta(t)) \right| \dt \\
& \ge \ln A_{4\chi} - 2 \int_0^\omega \gamma(t)\e^{u_3-u_4} \dt
\ge \ln A_{4\chi} -2(\bar\mu+\bar\eta)\omega.
\end{split}
\]
Finally, integrating the first equation of~\eqref{eq:sistema-aplic-Mawhin-lambda} in $[0,\omega]$ and using~\eqref{eq:ui_le_ln(Ai)} and~\eqref{eq:qualquer-coisa-4}, we obtain
\[
\begin{split}
\int_0^\omega  \Lambda(t)\e^{-u_1}+\eta(t)\e^{u_4-u_1}\dt
& = \int_0^\omega \beta(t)\,\varphi\left(\e^{u_1(t)},w(t),\e^{u_3(t)}\right)\e^{-u_1(t)} +\mu(t) \dt \\
& = \int_0^\omega \beta(t)\,
\dfrac{\varphi\left(\e^{u_1(t)},w(t),\e^{u_3(t)}\right)}{\e^{u_1(t)+u_3(t)}}\e^{u_3(t)} +\mu(t) \dt \\
& \le \left( \bar\beta c_2 A_{3\xi}\e^{-2(\bar\mu+\bar\gamma)\omega}+ \bar\mu \right) \omega,
\end{split}
\]
and thus
\[
\begin{split}
u_1(t)
& \le u_1(\xi_1) + \int_0^\omega |u_1'(t)| \dt \\
& = u_1(\xi_1) + \lambda \int_0^\omega
\left| \Lambda(t)\e^{-u_1}-\beta(t)\,\frac{C(w)}{w} \e^{u_3} -\mu(t)+\eta(t)\e^{u_4-u_1} \right| \dt \\
& \le \ln A_{1\xi} + 2 \int_0^\omega \Lambda(t)\e^{-u_1}+\eta(t)\e^{u_4-u_1} \dt \\
& \le \ln A_{1\xi} +2\left( \bar\beta c_2 A_{3\xi}\e^{-2(\bar\mu+\bar\gamma)\omega}+ \bar\mu \right) \omega
\end{split}
\]
and also
\[
\begin{split}
u_1(t)
& \ge u_1(\chi_1) - \int_0^\omega |u_1'(t)| \dt \\
& = u_1(\chi_1) - \lambda \int_0^\omega
\left| \Lambda(t)\e^{-u_1}-\beta(t)\,\frac{C(w)}{w} \e^{u_3} -\mu(t)+\eta(t)\e^{u_4-u_1} \right| \dt \\
& \ge \ln A_{1\chi} - 2 \int_0^\omega \Lambda(t)\e^{-u_1}+\eta(t)\e^{u_4-u_1} \dt \\
& \ge \ln A_{1\chi} - 2\left( \bar\beta c_2 A_{3\xi}\e^{-2(\bar\mu+\bar\gamma)\omega}+ \bar\mu \right) \omega.
\end{split}
\]

Consider the algebraic system
\begin{equation}\label{eq:algebrico}
\begin{cases}
\bar\Lambda \e^{-u_1}-\bar\beta \varphi(\e^{u_1},w,\e^{u_3}) \e^{-u_1} -\bar\mu+\bar\eta\e^{u_4-u_1} =0\\
\bar\beta \varphi(\e^{u_1},w,\e^{u_3}) \e^{-u_2} -\bar\mu-\bar\epsilon =0\\
\bar\epsilon \e^{u_2-u_3} -\bar\mu-\bar\gamma =0\\
\bar\gamma\e^{u_3-u_4}-\bar\mu-\bar\eta=0
\end{cases}.
\end{equation}
Multiplying the first equation by $\e^{u_1}$, the second by $\e^{u_2}$, the third by $\e^{u_3}$
and the fourth equation by $\e^{u_4}$ and adding the equations we conclude that any solution of this equation verifies
    $$w=\frac{\bar\Lambda}{\bar\mu}.$$
Moreover, we conclude by simple computations that the solution of system~\eqref{eq:algebrico}  verifies
\begin{equation}\label{eq:sol_e^u_2_e^u_4}
  \e^{u_2}=\frac{\bar\mu+\bar\gamma}{\bar\epsilon} \e^{u_3}
=\frac{(\bar\mu+\bar\gamma)(\bar\mu+\bar\eta)}{\bar\epsilon\bar\gamma} \e^{u_4}
\end{equation}
and also
\begin{equation}\label{eq:sol_e^u_1}
\e^{u_1} = \frac{\bar\Lambda}{\bar\mu}- \frac{(\bar\mu+\bar\gamma)(\bar\mu+\bar\epsilon)(\bar\mu+\bar\eta) -\bar\epsilon\bar\gamma\bar\eta}{\bar\epsilon\bar\mu(\bar\mu+\bar\eta)}\e^{u_3}.
\end{equation}
Thus, by the second equation in~\eqref{eq:algebrico} we get
\begin{equation}\label{eq:sol-e^u_3}
\dfrac{\bar\epsilon\bar\beta}{\bar\mu+\bar\gamma} \, \varphi\left(\bar\Lambda/\bar\mu - d \e^{u_3}, \bar\Lambda/\bar\mu, \e^{u_3}\right)\e^{-u_3}-(\bar\mu+\bar\epsilon)=0,
\end{equation}
where
\[
d=\frac{(\bar\mu+\bar\gamma)(\bar\mu+\bar\epsilon)(\bar\mu+\bar\eta) -\bar\epsilon\bar\gamma\bar\eta}{\bar\epsilon\bar\mu(\bar\mu+\bar\eta)}.
\]
By Lemma~\ref{lemma:unique-sol}, ~\eqref{eq:sol-e^u_3} has a unique solution. Therefore, by~\eqref{eq:sol_e^u_2_e^u_4} and~\eqref{eq:sol_e^u_1} we conclude that the algebraic system~\eqref{eq:algebrico} has a unique solution. Denote this solution by $p^*=(p_1^*,p_2^*,p_3^*,p_4^*)$. Let $M_0>0$ be such that $|p_1^*|+|p_2^*|+|p_3^*|+|p_4^*|<M_0$ and let
{\small{$$M_1=\max\{|\ln A_{1\xi} +2\left( \bar\beta c_2 A_{3\xi}\e^{-2(\bar\mu+\bar\gamma)\omega}+ \bar\mu \right) \omega|,
  |\ln A_{1\chi} - 2\left( \bar\beta c_2 A_{3\xi}\e^{-2(\bar\mu+\bar\gamma)\omega}+ \bar\mu \right) \omega|\},$$}}
$$M_2=\max\{|\ln A_{2\xi} +2(\bar \mu + \bar\epsilon)\omega|, |\ln A_{2\chi} - 2(\bar \mu + \bar\epsilon)\omega|\},$$
$$M_3=\max\{|\ln A_{3\xi} + 2(\bar \mu + \bar\gamma)\omega|, |\ln A_{3\chi} -2(\bar \mu + \bar\gamma)\omega|\},$$
and
$$M_4=\max\{|\ln A_{4\xi} +2(\bar\mu+\bar\eta)\omega|, |\ln A_{4\chi} -2(\bar\mu+\bar\eta)\omega|\}.$$
Define
    $$M=M_0+M_1+M_2+M_3+M_4.$$
We will apply Mawhin's Theorem in the open set
    $$\Omega=\{(u_1,u_2,u_3,u_4) \in X: \|(u_1,u_2,u_3,u_4)\|<M \}.$$
Let $u \in \partial \Omega \cap \ker \mathcal L = \partial \Omega \cap \R^4$. Then $u$ is a constant function that we can identify with the vector $(u_1,u_2,u_3,u_4) \in \R^4$ with $\|u\|=M$ and
\[
Q \mathcal N u :=
\left[
\begin{array}{c}
F_1(u)\\
F_2(u)\\
F_3(u)\\
F_4(u)
\end{array}
\right]
=
\left[
\begin{array}{c}
\bar\Lambda \e^{-u_1}-\bar\beta \varphi(\e^{u_1},w,\e^{u_3}) \e^{-u_1} -\bar\mu+\bar\eta\e^{u_4-u_1}\\
\bar\beta \varphi(\e^{u_1},w,\e^{u_3}) \e^{-u_2}  -\bar\mu-\bar\epsilon\\
\bar\epsilon \e^{u_2-u_3} -\bar\mu-\bar\gamma\\
\bar\gamma\e^{u_3-u_4}-\bar\mu-\bar\eta
\end{array}
\right] \ne 0.
\]
We conclude that
\[
\begin{split}
\deg(\Id Q \mathcal N, \partial \Omega \cap \ker L, (0,0,0,0))
& = \sum_{x \in (Id Q \mathcal N)^{-1}(0,0,0,0)} \text{sign} \det
d_x(\Id Q \mathcal N) \\
& = \text{sign} \det d_{p^*}(\Id Q \mathcal N) \\
& = \text{sign} \det \mathcal M,
\end{split}
\]
where $\mathcal M$ is the matrix in~\eqref{eq:matrixM}. By hypothesis \ref{cond-C3}) we have $\det \mathcal M \ne 0$. Thus
$$\deg(\Id Q \mathcal N u, \partial \Omega \cap \ker L, (0,0,0,0)) \ne 0.$$
According to Mawhin's continuation theorem, we conclude that equation $\mathcal L x = \mathcal N x$ has at least one solution in $D \cap \bar U$. Therefore, in the hypothesis of the theorem, we conclude that system~\eqref{eq:ProblemaPrincipal} has at least one $\omega$-periodic solution and the result follows.
\end{proof}

The following corollary shows that, when $\varphi$ does not depend explicitly on the total population, the condition $\det \mathcal M \ne 0$ is always satisfyed.

\begin{corollary}
  Let $\varphi(S,N,I)=\psi(S,I)$ and assume that it satisfies conditions P\ref{Cond-P1}) to P\ref{Cond-P5}). If $\mathcal R_0>1$ then system~\eqref{eq:ProblemaPrincipal} has an endemic periodic solution of period $\omega$.
\end{corollary}

\begin{proof}
We are assuming that $\mathcal R_0>1$ and thus we have condition~\ref{cond-C0}) in Theorem~\ref{teo:main}. Some computations yield
\begin{equation}\label{eq:no-total-pop-cond-ne-0}
\det \mathcal M
=  -\frac{(\bar\eta+\bar\mu)(\bar\gamma+\bar\mu)}{q}\left( \bar\eta s \frac{\partial \phi}{\partial S}(p,\bar\Lambda/\bar\mu,r)
 + \bar\mu r
 \frac{\partial \phi}{\partial I}(p,\bar\Lambda/\bar\mu,r)\right).
\end{equation}
By P\ref{Cond-P4}) and~\eqref{eq:no-total-pop-cond-ne-0} we have $\det \mathcal M \ne 0$. Thus, condition~\ref{cond-C3}) in Theorem~\ref{teo:main} holds. The result follows from Theorem~\ref{teo:main}.
\end{proof}

The following is an immediate corollary of the previous one.

\begin{corollary}[Simple incidence functions]\label{cor:simple}
  Let $\varphi(S,N,I)=SI$. If $\mathcal R_0>1$ then system~\eqref{eq:ProblemaPrincipal} has an endemic periodic solution of period $\omega$.
\end{corollary}

In~\cite{Zhang-Liu-Teng-AM-2012} it is discussed the existence of periodic orbits for a model with mass-action incidence and disease induced mortality. When the disease induced mortality is set to zero (letting $\alpha\equiv 0$), the model considered in~\cite{Zhang-Liu-Teng-AM-2012} becomes a particular case of ours. For the no disease induced mortality case, Corollary~\ref{cor:simple} improves the main result in~\cite{Zhang-Liu-Teng-AM-2012}.

The next corollary shows that, in the case of Michaelis-Menten incidence, the condition $\det \mathcal M \ne 0$ is also always satisfied.

\begin{corollary}[Michaelis-Menten incidence functions]\label{cor:Mic-Ment}
  Let $\varphi(S,N,I)=\frac{C(N)}{N}SI$ and assume that $N \mapsto C(N)$ is continuously differentiable and positive and that $N \to C(N)/N$ is non-increasing. If $\mathcal R_0>1$ then system~\eqref{eq:ProblemaPrincipal} has an endemic periodic solution of period $\omega$.
\end{corollary}

\begin{proof}
In this case we have
\[
\begin{split}
\det \mathcal M
& =
-\dfrac{\bar\beta(\eta+\mu)(\gamma+\mu)}{q}
\left(
\dfrac{\partial \phi}{\partial N}(p,\bar\Lambda/\bar\mu,r)
\left(\mu r+\eta s+\mu q+\mu s\right)\right.\\
& \left. \quad
+ \eta s \dfrac{\partial \phi}{\partial S}(p,\bar\Lambda/\bar\mu,r)
+ \mu r \dfrac{\partial \phi}{\partial I}(p,\bar\Lambda/\bar\mu,r)
\right)\\
& =
-\dfrac{\bar\beta(\eta+\mu)(\gamma+\mu)}{q}
\left(
\dfrac{C'(\bar\Lambda/\bar\mu)}{\bar\Lambda/\bar\mu}
pr\left(\mu r+\eta s+\mu q+\mu s\right)\right.\\
& \quad \left.-\dfrac{C(\bar\Lambda/\bar\mu)}{\bar\Lambda^2/\bar\mu^2}
pr\left(\mu r+\eta s+\mu q+\mu s\right)+\dfrac{C(\bar\Lambda/\bar\mu)}{\bar\Lambda/\bar\mu}r(s\eta+p\mu)
\right).
\end{split}
\]
Since $p<\bar\Lambda/\bar\mu$ and $r+q+s=\bar\Lambda/\bar\mu-p$, we have
\[
\begin{split}
\det \mathcal M
& =
-\dfrac{\bar\beta(\eta+\mu)(\gamma+\mu)}{q}
\left(
\dfrac{C'(\bar\Lambda/\bar\mu)}{\bar\Lambda/\bar\mu}
pr\left(\mu r+\eta s+\mu q+\mu s\right)\right.\\
& \quad \left. +\dfrac{C(\bar\Lambda/\bar\mu)}{\bar\Lambda/\bar\mu}r(p^2\bar\mu^2/\bar\Lambda^2+ \eta s(1-p/(\bar\Lambda/\bar\mu)))
\right) < 0.
\end{split}
\]
Thus, $\det \mathcal M \ne 0$ and the result follows.
\end{proof}

We obtain again, as a particular case of Corollary~\ref{cor:Mic-Ment}, the result in Corollary~\ref{cor:simple}.

\section{A family of examples}

In this section we restrict our attention to the family of periodic systems
\begin{equation}\label{eq:ProblemaEndemic}
\begin{cases}
S'=\Lambda-\beta[1+a\cos(2\pi t +\phi)]\,SI-\mu S\\
E'=\beta[1+a\cos(2\pi t +\phi)]\,SI-\mu E+\epsilon E\\
I'=\epsilon E -\mu I-\gamma I \\
R'=\gamma I-\mu R \\
N=S+E+I+R
\end{cases}.
\end{equation}
In~\cite{Bacaer-BMB-2007} (see equation (51)), it was shown that, for small $b$, we have
    \begin{equation}\label{eq:approx-threshold-periodic}
      \mathcal R_0=\dfrac{\beta\eps}{(\mu+\eps)(\mu+\gamma)}+\dfrac{\beta\eps b^2/2}{4\pi^2+(2\mu+\eps+\gamma)^2}+o(b^2).
    \end{equation}

Set $\Lambda=\mu=2$, $\eps=1$, $\gamma=0.02$ and consider the following initial conditions $S_0=E_0=I_0=R_0=0.1$ (black lines). We assume that there is no loss of immunity and let $\eta=0$. To consider a periodic case, we begin by setting $b=0.1$, $\beta=5.9$ and $\varphi=0$ in~\eqref{eq:ProblemaEndemic}.

Using approximation~\eqref{eq:approx-threshold-periodic}, we have the estimate $\mathcal R_0 \approx 0,9741 < 1$ and we conclude that the disease goes to extinction. We can see this in the right-hand side of figure~\ref{fig_estimate}.
\begin{figure}[h!]
  \begin{minipage}[b]{.495\linewidth}
    \includegraphics[width=\linewidth]{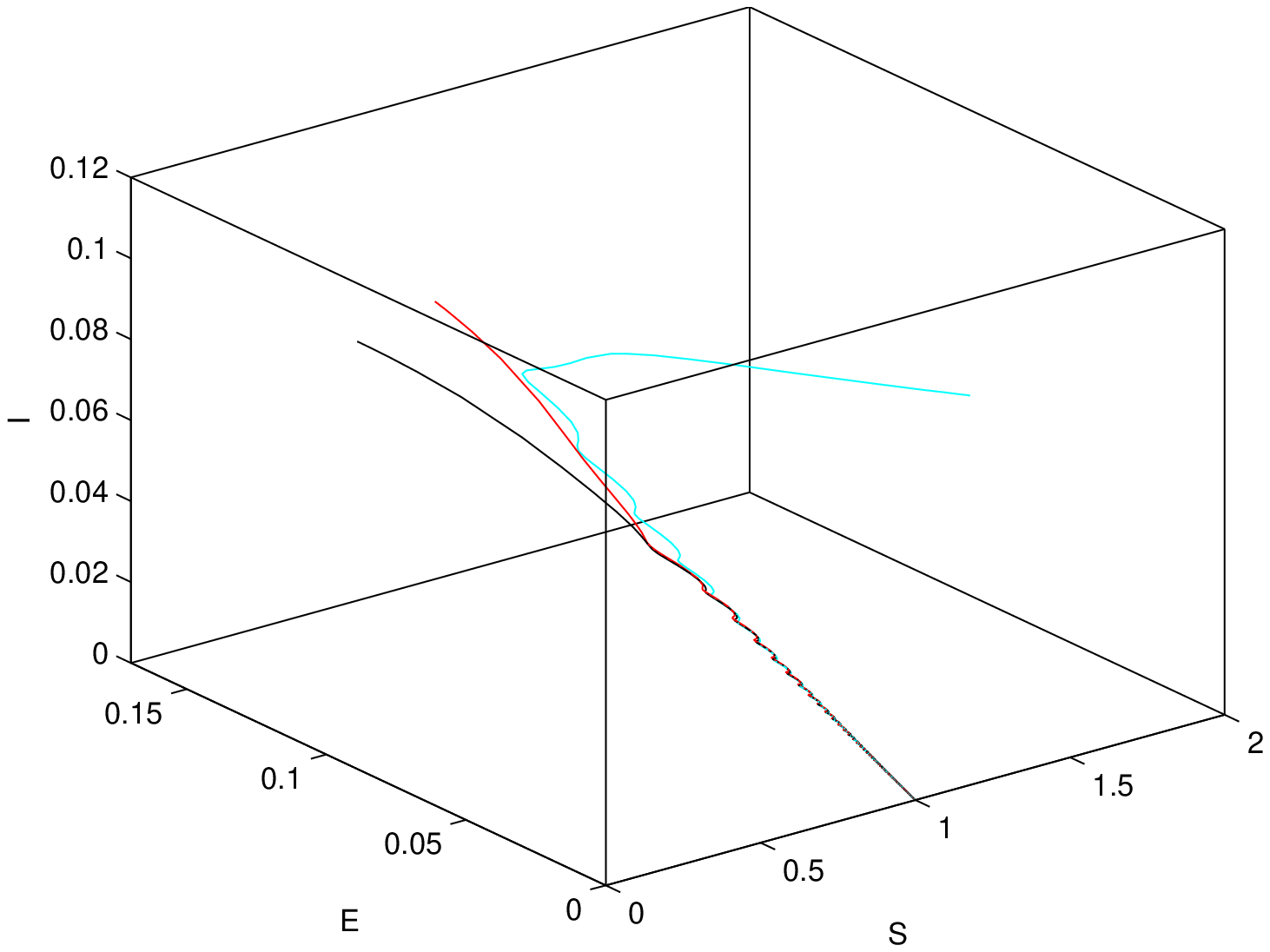}
  \end{minipage} \hfill
  \begin{minipage}[b]{.495\linewidth}
        \includegraphics[width=\linewidth]{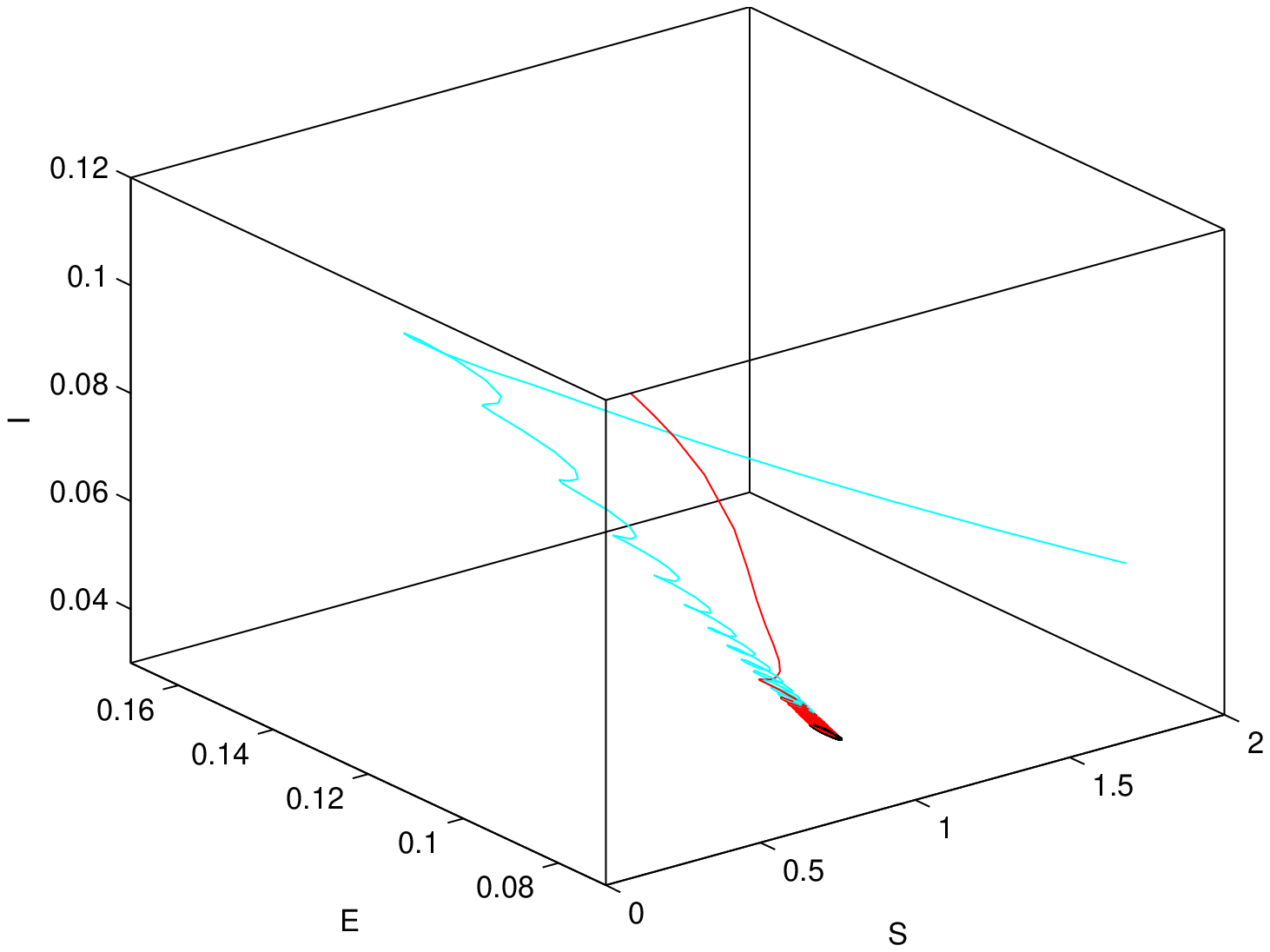}
  \end{minipage}
    \caption{Disease-free case and endemic case for $b=0.1$.}
      \label{fig_estimate}
\end{figure}

Changing $b$ to $0.6$, estimate~\eqref{eq:approx-threshold-periodic} gives $\mathcal R_0 \approx 0.9900 < 1$ an increased $\cR_0$ that still makes the disease go to extinction. In fact, a plot for this case can be seen on the left-hand side of figure~\ref{fig_Periodic} where we can see that all trajectories approach the disease-free equilibrium $e^*=(1,0,0,0)$ and thus that the disease goes to extinction.

On the right-hand side of figure~\ref{fig_estimate}, we let $b=0.1$ and $\beta=6.9$. We can see that the disease persists and that all trajectories approach an endemic periodic orbit. In this case, the approximate formula~\eqref{eq:approx-threshold-periodic}, gives $\mathcal R_0=1.13915>1$ and we also have $\beta^\ell\eps^\ell\Lambda^\ell/((\mu+\eps)^u(\mu+\gamma)^u\mu^u)=1.02475>1$. Both Corollary~\ref{cor:simple} and the main result in~\cite{Zhang-Liu-Teng-AM-2012} confirm the existence of an endemic periodic orbit.

If we increase the oscillations and set $b=0.6$, the approximate formula~\eqref{eq:approx-threshold-periodic}, gives $\mathcal R_0=1.15782>1$. In this case Corollary~\ref{cor:simple} still allows us to conclude that there is an endemic periodic orbit. This conclusion is not possible using the result in~\cite{Zhang-Liu-Teng-AM-2012} since in this case $\beta^\ell\eps^\ell\Lambda^\ell/((\mu+\eps)^u(\mu+\gamma)^u\mu^u)=0.455446<1$.
In the right-hand side of figure~\ref{fig_Periodic} we can see that the disease persists and that all trajectories approach an endemic periodic orbit. Note that the red and cyan lines correspond respectively to solutions with the following initial conditions: $S_0=0.08$, $E_0=0.07$, $I_0=0.12$, $R_0=0.13$ and $S_0=1.99$, $E_0=0.09$, $I_0=0.05$, $R_0=0.25$.

\begin{figure}[h!]
  \begin{minipage}[b]{.495\linewidth}
    \includegraphics[width=\linewidth]{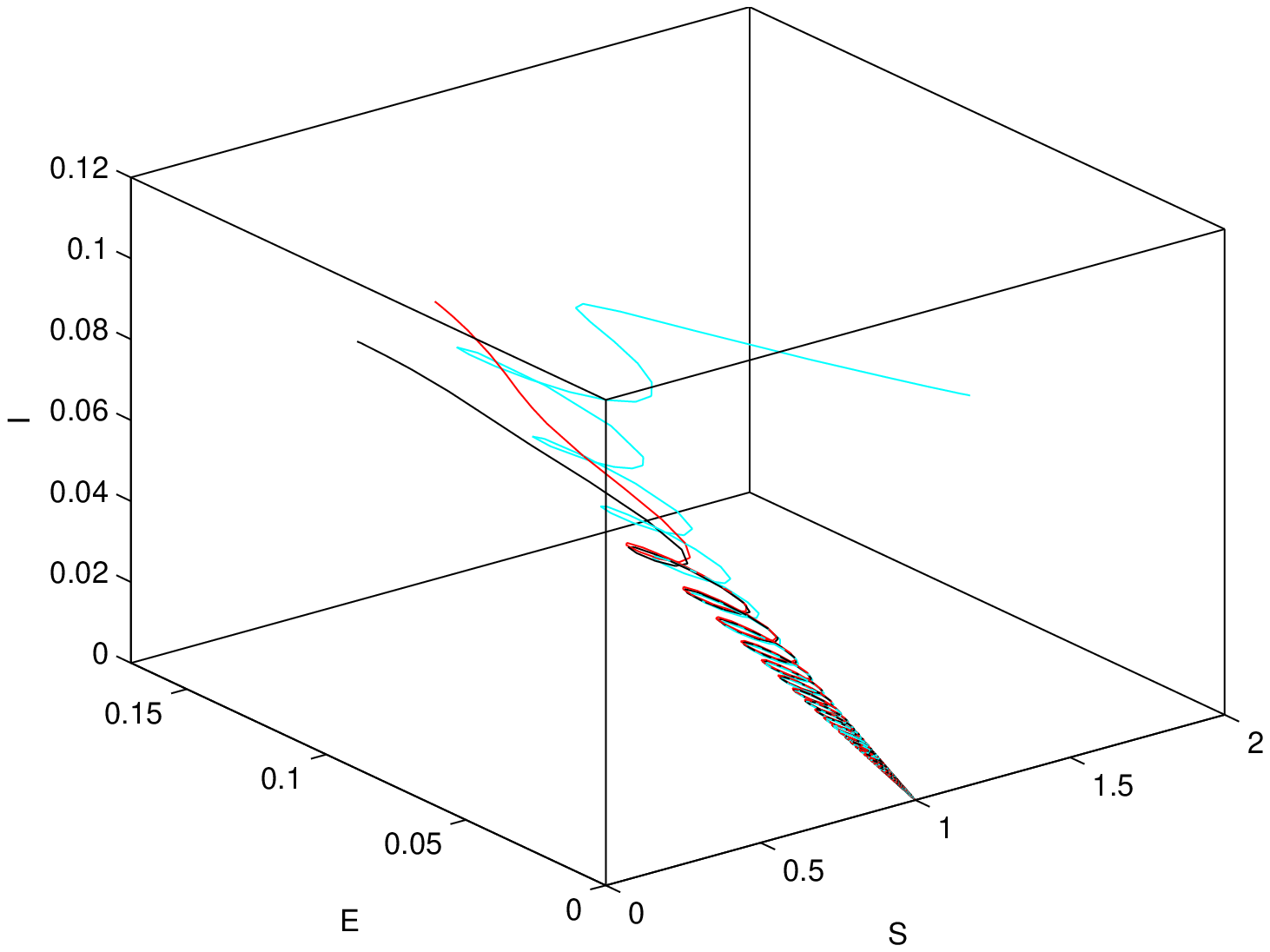}
  \end{minipage} \hfill
  \begin{minipage}[b]{.495\linewidth}
        \includegraphics[width=\linewidth]{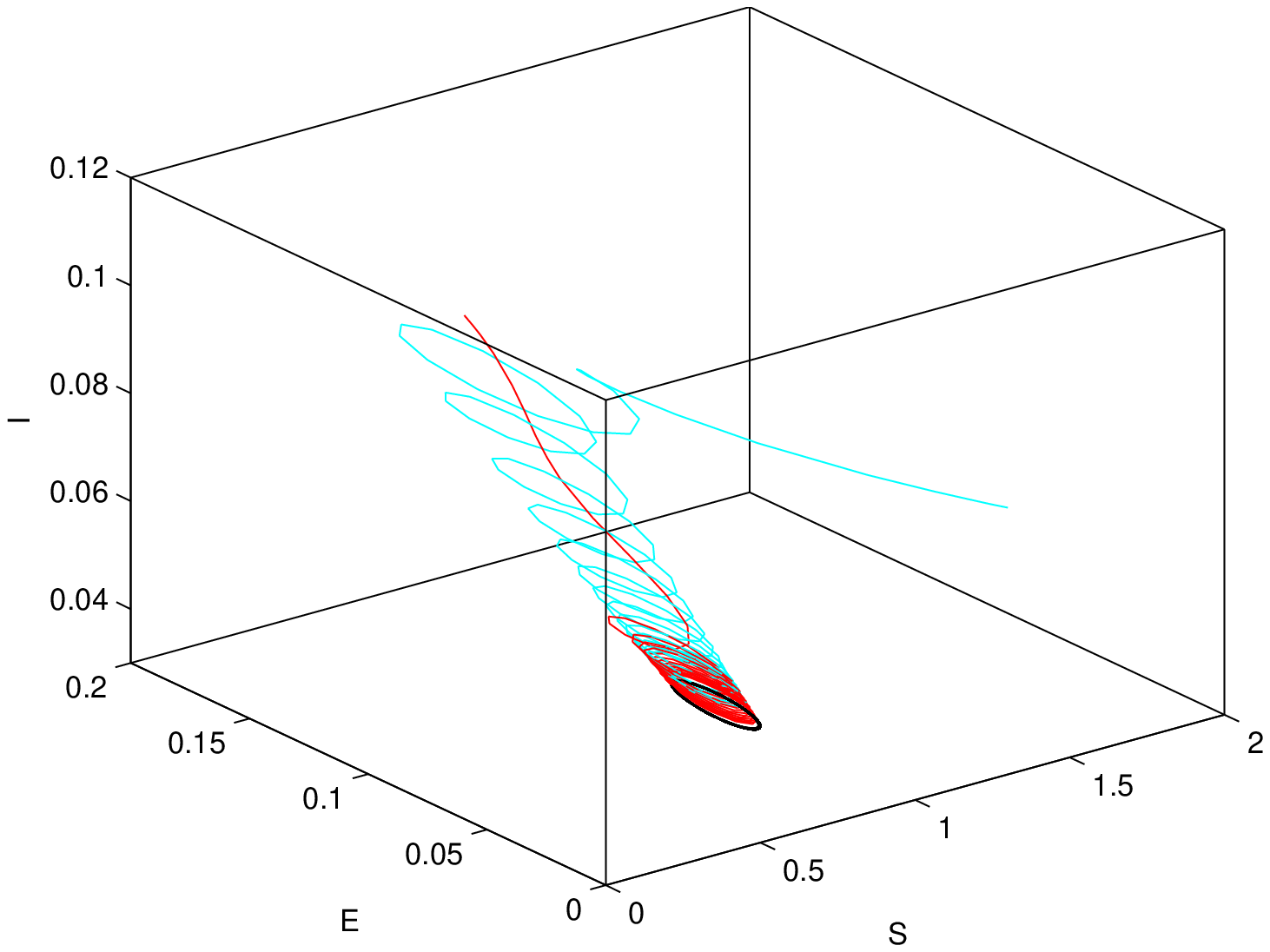}
  \end{minipage}
    \caption{Disease Free Case and Endemic Case for $b=0.6$.}
      \label{fig_Periodic}
\end{figure}

In figures~\ref{fig_Periodic_FREE} and~\ref{fig_Periodic_ENDE} we present the trajectories of the infectives and the susceptibles for the situations described in figure~\ref{fig_Periodic}.

\begin{figure}[h!]
  \begin{minipage}[b]{.495\linewidth}
    \includegraphics[width=\linewidth]{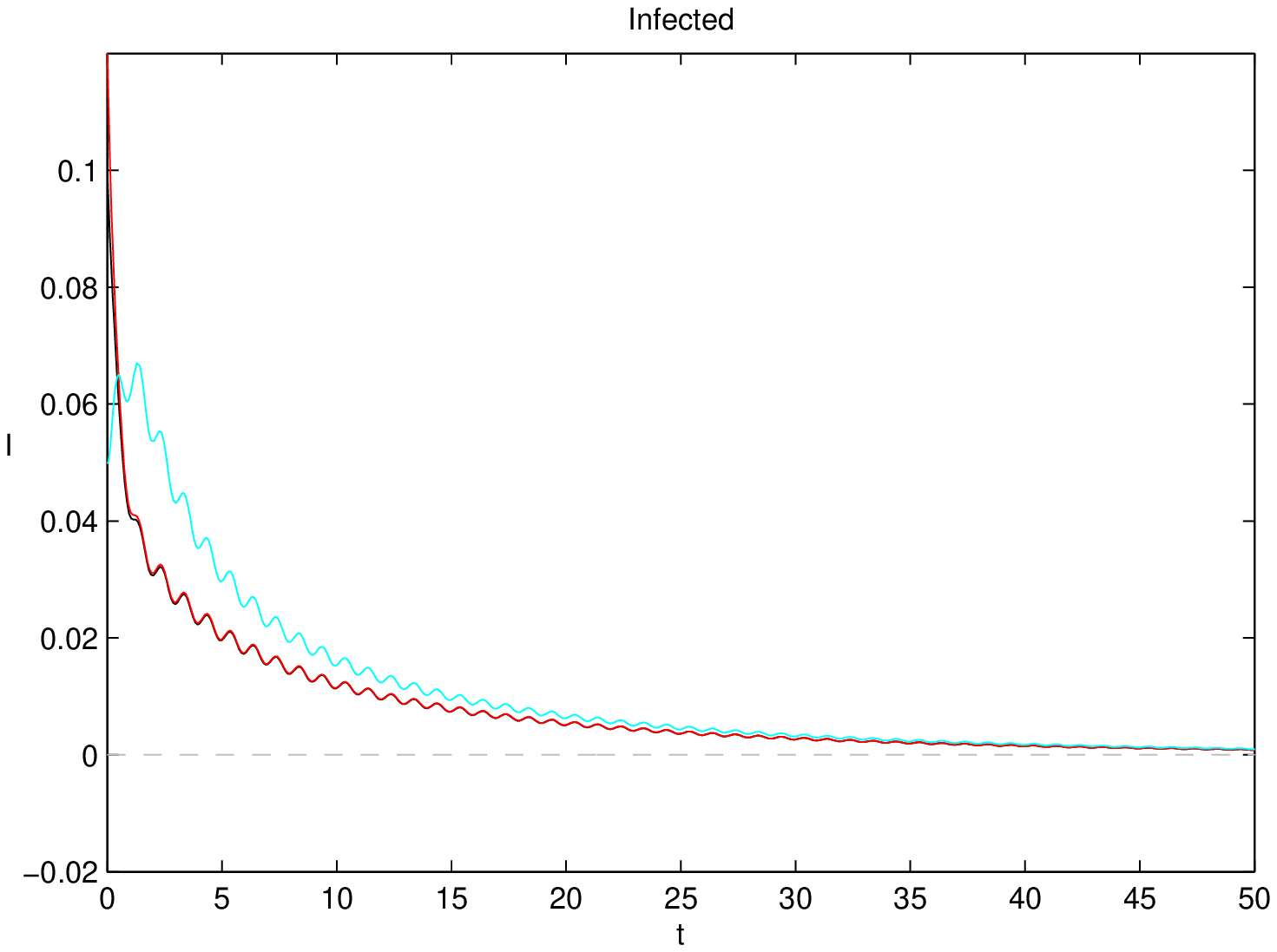}
  \end{minipage} \hfill
  \begin{minipage}[b]{.495\linewidth}
        \includegraphics[width=\linewidth]{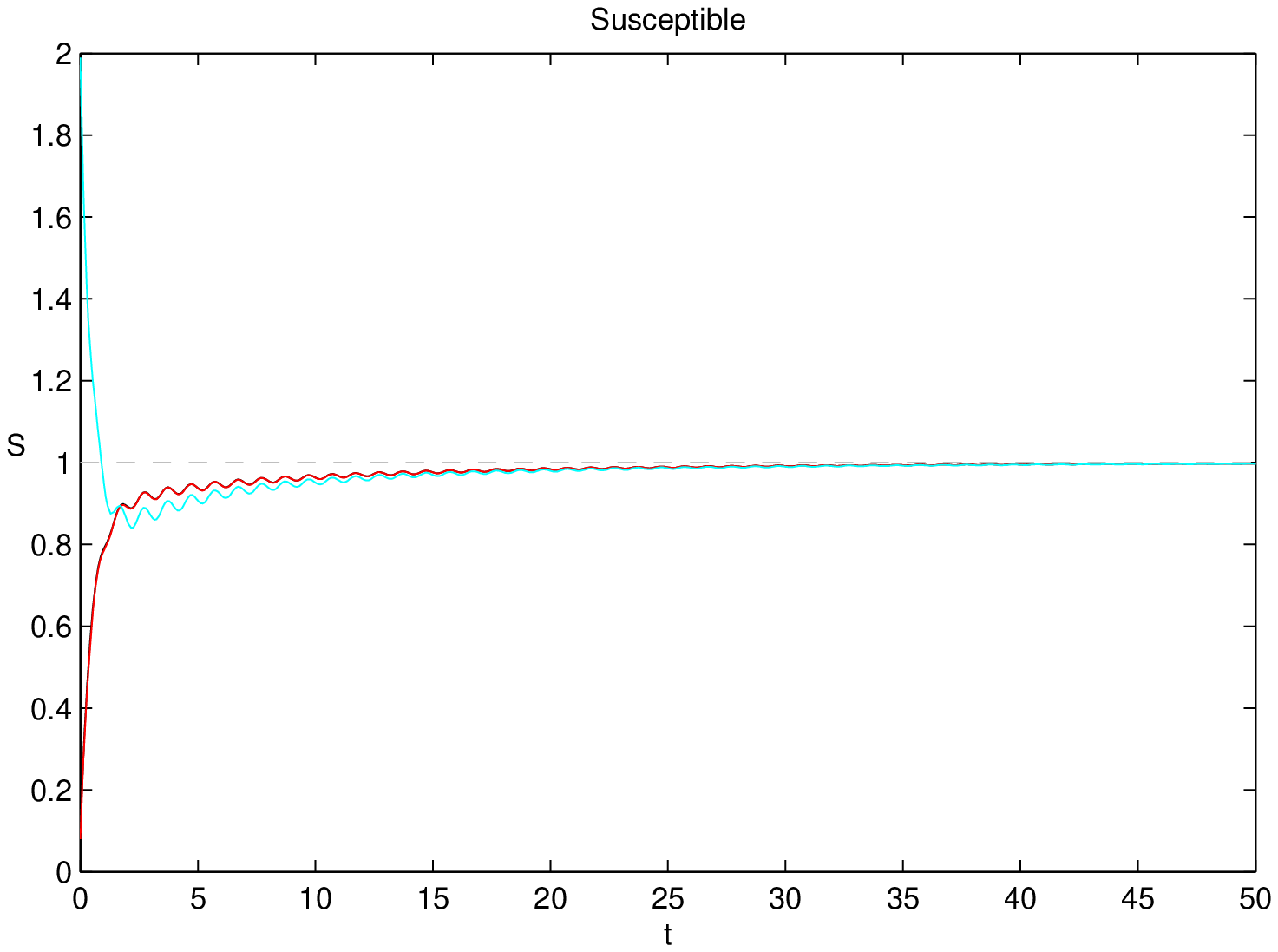}
  \end{minipage}
    \caption{Disease Free Case.}
      \label{fig_Periodic_FREE}
\end{figure}
\begin{figure}[h!]
  \begin{minipage}[b]{.495\linewidth}
    \includegraphics[width=\linewidth]{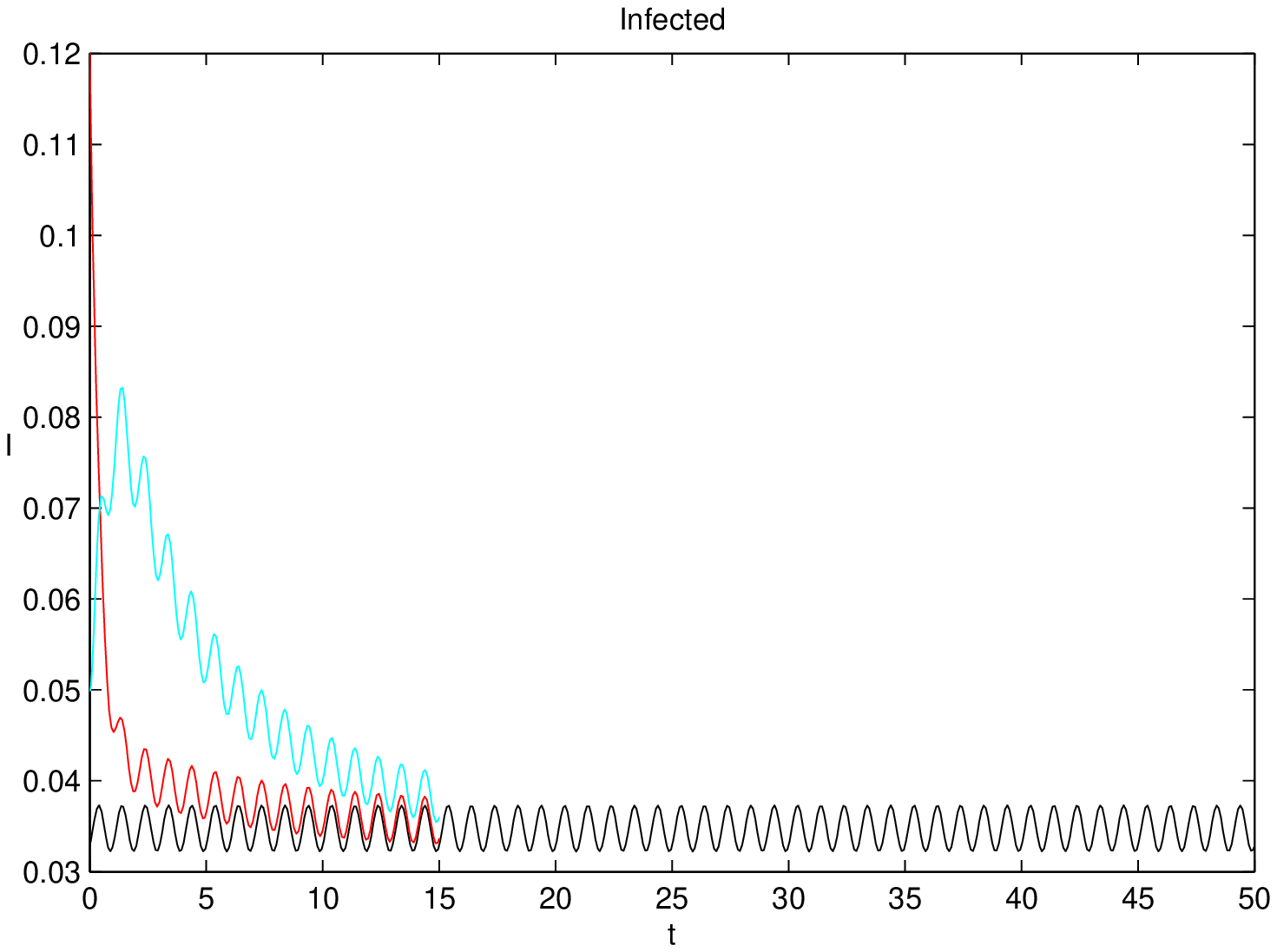}
  \end{minipage} \hfill
  \begin{minipage}[b]{.495\linewidth}
        \includegraphics[width=\linewidth]{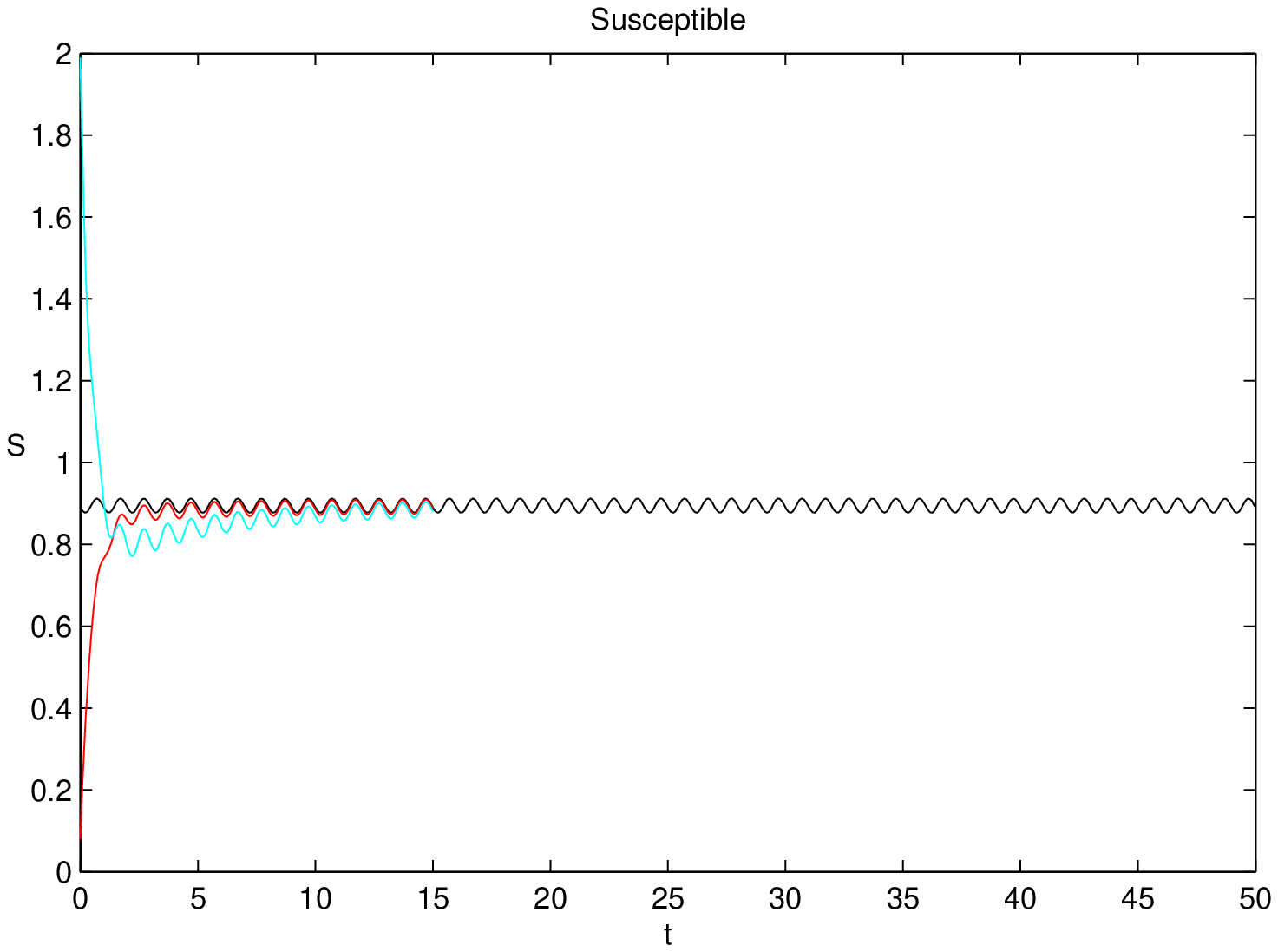}
  \end{minipage}
    \caption{Endemic Case.}
      \label{fig_Periodic_ENDE}
\end{figure}

\end{document}